\documentclass[a4paper,reqno, 11pt]{amsart}

\usepackage[utf8]{inputenc}
\usepackage{microtype}

\usepackage{amsthm,amstext,amsmath,amssymb, graphicx, enumerate, enumitem, color}
\usepackage[usenames,dvipsnames,svgnames,table]{xcolor}
\usepackage[foot]{amsaddr}
\usepackage{todonotes}
\usepackage{thmtools}

\colorlet{mylinkcolor}{violet}
\colorlet{mycitecolor}{YellowOrange}
\colorlet{myurlcolor}{Aquamarine}
\usepackage[unicode=true]{hyperref}
\hypersetup{
    pdftitle={Excluding a ladder},
    pdfauthor={Tony Huynh, Gwena\"el Joret, Piotr Micek, Michał T.\ Seweryn, Paul Wollan},
    colorlinks,
    linkcolor={red!50!black},
    citecolor={blue!50!black},
    urlcolor={blue!80!black}
}

\newtheorem{theorem}{Theorem}
\newtheorem{conjecture}[theorem]{Conjecture}
\newtheorem{corollary}[theorem]{Corollary}

\newtheorem{observation}[theorem]{Observation}
\newtheorem{lemma}[theorem]{Lemma}

\newcounter{claim_nb}[theorem]
\newtheorem{claim}[claim_nb]{Claim}

\usepackage{mathtools}

\DeclarePairedDelimiter\ceil{\lceil}{\rceil}
\DeclarePairedDelimiter\set{\{}{\}}

\newcommand{\calC}{\mathcal{C}}

\newcommand{\calT}{\mathcal{T}}

\DeclareMathOperator\td{td}

\let\le\leqslant
\let\ge\geqslant
\let\leq\leqslant
\let\geq\geqslant

\let\subset\subseteq

\let\epsilon\varepsilon

\renewenvironment{enumerate}{\begin{enumorig}[label=\textup{(\roman*)}, noitemsep, topsep=2pt plus 2pt, labelindent=.2em, leftmargin=*, widest=iii]}{\end{enumorig}}

\renewenvironment{itemize}{\begin{itemorig}[label=\textbullet, noitemsep, topsep=2pt plus 2pt, labelindent=.5em, labelsep=.5em, leftmargin=*]}{\end{itemorig}}

\makeatletter
\let\old@setaddresses\@setaddresses
\def\@setaddresses{\bigskip\bgroup\parindent 0pt\let\scshape\relax\old@setaddresses\egroup}
\makeatother

\newcommand{\zB}{\mathcal B}

\newcommand{\tdtwo}{\td_2}
\newcommand{\chicir}{\chi_{\textnormal{cc}}}

\newcommand{\ignore}[1]{}

\newenvironment{cproof}
{\begin{proof}
 [Proof.]
 \vspace{-1.5\parsep}
}
{ \end{proof}}

\begin{document}
\title{Excluding a ladder}

\author[T.~Huynh]{Tony Huynh}
\address[T.~Huynh]{School of Mathematics \\
Monash University \\
Melbourne, Australia}
\email{tony.bourbaki@gmail.com}

\author[G.~Joret]{Gwena\"el  Joret}
\address[G.~Joret]{Computer Science Department \\
  Universit\'e Libre de Bruxelles\\
  Brussels, Belgium}
\email{gjoret@ulb.ac.be}

\author[P.~Micek]{Piotr Micek}
\address[P.~Micek, M.T.~Seweryn]{Theoretical Computer Science Department\\
  Faculty of Mathematics and Computer Science, Jagiellonian University \\
  Krak\'ow, Poland}
\email{piotr.micek@tcs.uj.edu.pl}

\author[M.T.~Seweryn]{Micha\l{} T.~Seweryn}
\email{michal.seweryn@tcs.uj.edu.pl}

\author[P.~Wollan]{Paul Wollan}
\address[P.~Wollan]{Department of Computer Science\\
University of Rome ``La Sapienza'' \\
Rome, Italy}
\email{wollan@di.uniroma1.it}

\thanks{T.\ Huynh is supported by the Australian Research Council.
G.\ Joret is supported by an ARC grant from the Wallonia-Brussels Federation of Belgium.
P.\ Micek was partially supported by the National Science Center of Poland under grant no.\ 2015/18/E/ST6/00299.
M.T.\ Seweryn was partially supported by Kartezjusz program WND-POWR.03.02.00-00-I001/16-01 funded by The National Center for Research and Development of Poland.}

\date{\today}



\begin{abstract}
A ladder is a $2 \times k$ grid graph.
When does a graph class $\mathcal{C}$ exclude some ladder as a minor?
We show that this is the case if and only if all graphs $G$ in $\mathcal{C}$ admit a proper vertex coloring with a bounded number of colors such that for every $2$-connected subgraph $H$ of $G$, there is a color that appears exactly once in $H$.
This type of vertex coloring is a relaxation of the notion of centered coloring, where for every connected subgraph $H$ of $G$, there must be a color that appears exactly once in $H$.
The minimum number of colors in a centered coloring of $G$ is the treedepth of $G$, and it is known that classes of graphs with bounded treedepth are exactly those that exclude a fixed path as a subgraph, or equivalently, as a minor.
In this sense, the structure of graphs excluding a fixed ladder as a minor resembles the structure of graphs without long paths.
Another similarity is as follows:
It is an easy observation that every connected graph with two vertex-disjoint paths of length $k$ has a path of length $k+1$.
We show that every \(3\)-connected graph which contains as a minor a union of sufficiently many vertex-disjoint copies of a \(2 \times k\) grid has a \(2 \times (k+1)\) grid minor.

Our structural results have applications to poset dimension.
We show that posets whose cover graphs exclude a fixed ladder as a minor have bounded dimension.
This is a new step towards the goal of understanding which graphs are unavoidable as minors in cover graphs of posets with large dimension.
\end{abstract}
\maketitle

\section{Introduction}
Graphs with no long paths are relatively well understood.
In particular, if a graph $G$ does not contain a path on \(k+1\) vertices as a subgraph, then
$G$ has a centered coloring with at most \(k\) colors.
Conversely, if $G$ has a centered coloring with at most \(k\) colors, then $G$ does not contain a path on \(2^k\) vertices.
Here, a \emph{centered coloring} of $G$ is a vertex coloring of $G$ such that
for every connected subgraph $H$ of $G$, some color is assigned to exactly one vertex of $H$.
The minimum number of colors used in a centered coloring of \(G\) is known as the {\em treedepth} of $G$, denoted $\td(G)$.

In this paper, we show an analogous result for graphs excluding a fixed ladder as a minor.
We show that such graphs can be characterized as graphs that admit a \(2\)-connected centered coloring with a bounded number of colors.
Here, a \emph{\(2\)-connected centered coloring} of a graph $G$ is a vertex coloring of $G$ such that for every connected subgraph $H$ of $G$ having no cutvertex, some color is assigned to exactly one vertex of $H$.\footnote{We remark that this definition is slightly different than the one given in the abstract but is equivalent. Indeed, if $H$ is connected with no cutvertex then $H$ is either a vertex, an edge, or is $2$-connected; edges will make sure that the coloring is proper.}
The minimum number of colors in a \(2\)-connected centered coloring of \(G\) is denoted \(\tdtwo(G)\).

Before stating our theorem formally, we introduce a related type of coloring.
A \emph{cycle centered coloring} of $G$ is a vertex coloring of $G$ such that for every subgraph $H$ of $G$ which is an edge or a cycle, some color is assigned to exactly one vertex of $H$.
The minimum number of colors in a cycle centered coloring of \(G\) is denoted \(\chicir(G)\).
While every \(2\)-connected centered coloring of a graph is cycle centered, the converse is not necessarily true.

Let \(L_k\) denote the ladder with $k$ rungs (that is, the $2\times k$ grid graph).
Our theorem for graphs excluding a ladder is as follows.

\begin{restatable}{theorem}{TheoremGraphsWithoutLongLadders}
  \label{thm:graphs-without-long-ladders}
  For every class \(\calC\) of graphs, the following properties are equivalent.
  \begin{enumerate}
    \item\label{itm:gwll-no-ladder} There exists an integer \(k \ge 1\) such that no graph in \(\calC\) has an \(L_k\) minor.
    \item\label{itm:gwll-2-c-coloring} There exists an integer \(m \ge 1\) such that \(\tdtwo(G) \le m\) for every graph \(G\) in \(\calC\).
    \item\label{itm:gwll-circ-coloring} There exists an integer \(c \ge 1\) such that \(\chicir(G) \le c\) for every graph \(G\) in \(\calC\).
  \end{enumerate}
\end{restatable}

A second contribution of this paper is as follows.
As is well known, every pair of longest paths in a connected graph intersect, or equivalently, if a connected graph contains two vertex disjoint paths of order \(k\), then it contains a path of order \(k+1\).
We show a generalization of this statement where paths are replaced with ladders, and `two' with `many'.

\begin{restatable}[Bumping a ladder]{theorem}{TheoremBumpingLadder}
  \label{thm:bumping-a-ladder}
  For every integer \(k \ge 1\) there exists an integer \(N \ge 1\) with the following property.
  Every \(3\)-connected graph \(G\) containing a union of \(N\) vertex-disjoint copies of \(L_k\) as a minor contains \(L_{k+1}\) as a minor.
\end{restatable}

Let us point out the following corollary of Theorem~\ref{thm:bumping-a-ladder}.
Robertson and Seymour~\cite{RS86} proved that for every fixed planar graph $H$ and every \(N \ge 1\), there exists \(N' \ge 1\) such that every graph $G$ not containing a union of \(N\) vertex-disjoint copies of \(H\) as a minor has a vertex subset $X$ with $|X|\leq N'$ such that $G-X$ has no $H$ minor.

\begin{corollary}\label{cor:bumping-a-ladder}
  For every integer \(k \ge 1\) there exists an integer \(N' \ge 1\) with the following property.
  Every \(3\)-connected graph \(G\) with no \(L_{k+1}\) minor has a vertex subset $X$ with $|X|\leq N'$ such that $G-X$ has no $L_k$ minor.
\end{corollary}

We remark that \(3\)-connectivity in Theorem~\ref{thm:bumping-a-ladder} is necessary.
See~Figure~\ref{fig:2-connected-counterexample}.
On the other hand, we expect that the dependence on $k$ is not.
We conjecture that there exists a constant \(N_0\) such that for every \(k\), Theorem~\ref{thm:bumping-a-ladder} holds true with \(N = N_0\).
For all we know, this might even be true with \(N_0 = 4\).
\begin{figure}[!h]
    \centering
    \includegraphics[scale=.75]{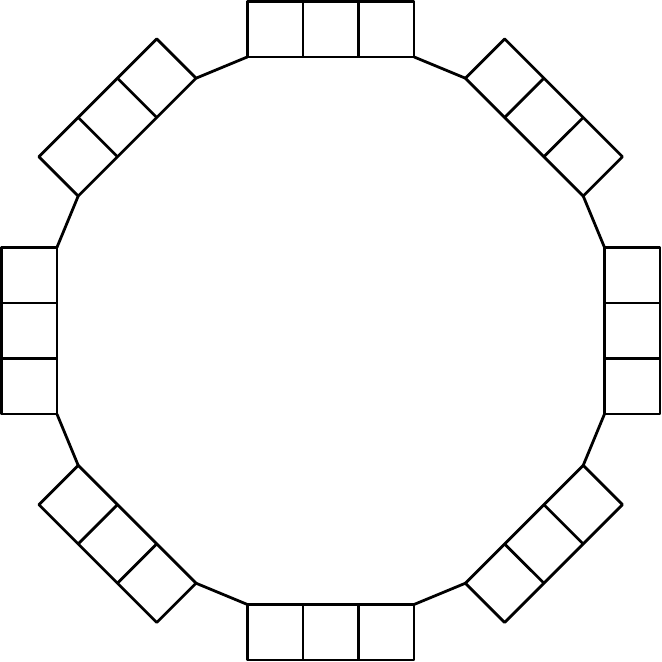}
    \caption{\label{fig:2-connected-counterexample}
    A $2$-connected graph with $8$ disjoint copies of $L_4$ and no $L_5$-minor.}
\end{figure}

We conclude this introduction with an application of our results to poset dimension.
Let \(P\) be a poset. For two elements \(x\) and \(y\) of \(P\), we say that \(y\) \emph{covers} \(x\) if \(x < y\) in \(P\) and there is no element \(z\) in \(P\) such that \(x < z < y\) in \(P\).
The \emph{cover graph} of a poset \(P\) is the graph on the ground set of \(P\) in which two vertices are adjacent if one of them covers the other in \(P\).
Informally, the cover graph of \(P\) is its Hasse diagram seen as an undirected graph.
A \emph{realizer} of \(P\) is a set \(\{\le_1, \dots, \le_d\}\) of linear orders on the ground set of \(P\) such that for any two elements \(x\) and \(y\) of \(P\), we have \(x \le y\) in \(P\) if and only if \(x \le_i y\) for every \(i \in \{1, \dots, d\}\).
The \emph{dimension} of \(P\), denoted \(\dim(P)\), is the least size of a realizer of \(P\).

We prove the following result, which shows the relevance of \(2\)-connected centered colorings for studying poset dimension.

\begin{restatable}{theorem}{TheoremPosetsTDTwo}
\label{thm:posets-tdtwo}
Let $P$ be a poset with cover graph $G$ and let $m=\tdtwo(G)$.
Then, $P$ has dimension at most \(2^{m+1} - 2\).
\end{restatable}

When combined with Theorem~\ref{thm:graphs-without-long-ladders}, the above theorem implies the following result.
Proving this result was one of our motivations for studying the structure of graphs with no long ladder minor.

\begin{corollary}\label{cor:posets-without-long-ladders}
For every integer \(k \ge 1\) there exists an integer \(d \ge 1\) such that every poset whose cover graph excludes \(L_k\) as a minor has dimension at most \(d\).
\end{corollary}

Let us say that a graph $H$ is {\em unavoidable} if the cover graph of every poset with large enough dimension contains $H$ as a minor.
Corollary~\ref{cor:posets-without-long-ladders} shows that every ladder is unavoidable.
Note that the class of unavoidable graphs is closed under taking minors (thus, fans are also unavoidable, etc.).
It is an open problem to obtain a full characterization of unavoidable graphs.
Besides ladders, the only other positive result known is that $K_4$ is unavoidable~\cite{JMTWW17, S20}.
As for negative results, a classic construction of Kelly~\cite{Kel81}, see Figure~\ref{fig:kelly}, shows that there are posets with unbounded dimension whose cover graphs are planar and have pathwidth $3$.
Note that every unavoidable graph must necessarily be a minor of some graph from Kelly's construction.
We conjecture that this is precisely the characterization of unavoidable graphs.

\begin{conjecture}\label{conj:unavoidable-graphs}
A graph $H$ is unavoidable if and only if $H$ is a minor of some graph from Kelly's construction.
\end{conjecture}
\begin{figure}[t]
    \centering
    \includegraphics{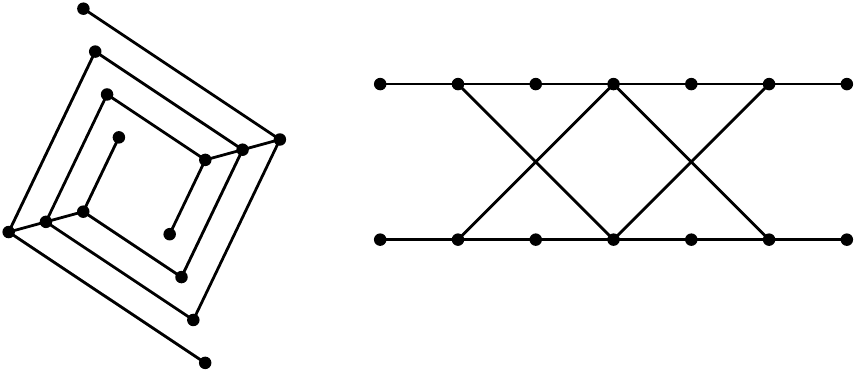}
    \caption{\label{fig:kelly}
    (Left) Hasse diagram of the poset of Kelly's construction of order $k$ for $k=4$. Its definition for an arbitrary order $k$ can be inferred from the figure.
    (Right) A free redrawing of its cover graph.}
\end{figure}

\begin{figure}[t]
  \centering
  \includegraphics{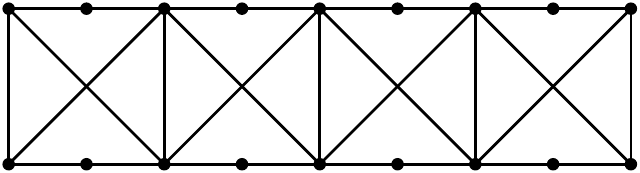}
  \caption{\label{fig:path_of_K4s}}
\end{figure}

Let us point out the following equivalent reformulation of Conjecture~\ref{conj:unavoidable-graphs}: A graph $H$ is unavoidable if and only if $H$ is a minor of some graph obtained by gluing copies of $K_4$s along edges in a path-like way, and subdividing all horizontal edges of the $K_4$s once; see Figure~\ref{fig:path_of_K4s}.
Indeed, every graph from Kelly's construction is a minor of a graph in the latter family (see Figure~\ref{fig:kelly} (right)), and vice versa.

We end with a short description of related results about poset dimension.
Unavoidable graphs are very restricted in nature, since they must be minors of some graph from Kelly's construction.
It is thus natural to consider posets excluding some fixed but arbitrary graph $H$ as a minor in their cover graphs, and ask: What kind of unavoidable structure can be found in these posets when dimension is large?
Several results have been obtained in the last decade when considering long chains as unavoidable structure.
Recall that a {\em chain} in a poset is a set of pairwise comparable elements.
The maximum length of a chain is the {\em height} of the poset.
In general, dimension is not bounded by any function of the height since there are height-$2$ posets with arbitrarily large dimension.
However, in 2014 Streib and Trotter~\cite{ST14} proved the surprising result that, if the cover graph of a poset is planar, then the poset's dimension is bounded from above by a function of its height (see~\cite{GS21+, JMW_PlanarPosets, KMT21+} for improved bounds).
This property was then shown to be true more generally if the cover graph excludes a fixed apex graph as a minor~\cite{JMMTWW}.
Then Walczak~\cite{Walczak17} showed that this remains true if any fixed graph $H$ is excluded from the cover graph as a minor, or even as a topological minor (see~\cite{MW15} for a short proof).
This was in turn generalized further: Posets with cover graphs belonging to a fixed graph class $\mathcal{C}$ with bounded expansion have dimension bounded from above by a function of their height~\cite{JMW18}.
A simple proof using a family of graph invariants called weak coloring numbers was given in~\cite{JMOdMW19}.
Finally, it was shown in~\cite{JMOdMW19} that nowhere dense graph classes\footnote{See Ne{\v{s}}et{\v{r}}il and Ossona de Mendez~\cite{NOdM-book} for background on classes with bounded expansion and on nowhere dense classes.} can be characterized in terms of dimension of bounded-height posets with cover graphs in the class.

Thus, long chains are unavoidable in posets with large dimension and well-behaved sparse cover graphs.
In the case of posets with planar cover graphs, Howard, Streib, Trotter, Walczak, and Wang~\cite{HSTWW19} strengthened this further by showing that two long incomparable chains can be found when dimension is large.
They conjectured that this remains true more generally when any fixed graph $H$ is forbidden as a minor from the cover graph.

The paper is organized as follows.
We introduce some basic definitions and notations in Section~\ref{sec:preliminaries}.
Then we prove Theorem~\ref{thm:graphs-without-long-ladders} in Section~\ref{sec:variant_center_colorings} and apply it in Section~\ref{sec:posets} to prove Theorem~\ref{thm:posets-tdtwo} about poset dimension.
Finally, we prove Theorem~\ref{thm:bumping-a-ladder} in Section~\ref{sec:bumping_ladders}.

\section{Preliminaries}
\label{sec:preliminaries}

In this section we recall some standard definitions.
For a graph $G$ and a subset $X$ of vertices, we denote by $G-X$ the subgraph of $G$ induced on the set of vertices $V(G) \setminus X$.  For a vertex $x \in V(G)$, we will use $G-x$ as shorthand notation for $G - \{x\}$.
A {\em cutvertex} of a graph \(G\) is a vertex \(v\) of \(G\) such that \(G - v\) has more connected components than \(G\).
A \emph{separation} of a graph \(G\) is a pair \((A_1, A_2)\) of vertex subsets in \(G\) such that \(A_1 \cup A_2 = V(G)\) and every edge of \(G\) has both endpoints in \(A_1\) or in \(A_2\).
The {\em order} of the separation is the number \(|A_1 \cap A_2|\), and a \emph{\(k\)-separation} is a separation of order at most \(k\).
A separation \((A_1, A_2)\) of \(G\) is \emph{trivial} if \(A_1 = V(G)\) or \(A_2 = V(G)\), and \emph{nontrivial} otherwise.
For \(k \ge 1\), we say that a graph is \emph{\(k\)-connected} if it has at least \(k+1\) vertices and does not admit a nontrivial \((k-1)\)-separation.

A \emph{subdivision} of a graph \(H\) is a graph obtained by replacing some edges of \(H\) with new paths between their endpoints such that none of the paths has an inner vertex in \(V(H)\) or on another new path.
An \emph{\(H\)-model} in a graph \(G\) is a function \(\phi\) which assigns to each vertex \(u \in V(H)\) a connected subgraph \(\phi(u)\) of \(G\), such that
\begin{enumerate}
  \item the graphs \(\phi(u)\) and  \(\phi(v)\) are vertex-disjoint for distinct vertices \(u\) and \(v\) of \(H\), and
  \item \(G\) has an edge between \(\phi(u)\) and \(\phi(v)\) for every edge \(u v \in E(H)\).
\end{enumerate}
A graph \(H\) is a \emph{minor} of a graph \(G\) if and only if there is an \(H\)-model in \(G\).

For two vertex subsets \(A\) and \(B\) in a graph \(G\), an \emph{\(A\)--\(B\) path} is a path with one endpoint in \(A\) and the second in \(B\), and with no internal vertex in \(A\cup B\).
For two vertices \(a\) and \(b\) of \(G\), by an \emph{\(a\)--\(b\) path} we mean  an \(\{a\}\)--\(\{b\}\) path.
If \(u\) and \(v\) are vertices on a path \(P\), then the only \(u\)--\(v\) path in \(P\) is denoted by \(u P v\).

For a positive integer \(k\), the \emph{ladder} \(L_k\) is the graph with vertex set \(\{1, 2\} \times \{1, \dots, k\}\) in which two vertices \((i, j)\) and \((i', j')\) are adjacent if \(|i - i'| + |j - j'| = 1\).
Since the maximum degree of \(L_k\) is at most \(3\), the ladder \(L_k\) is a minor of a graph \(G\)
if and only if \(G\) has a subgraph isomorphic to a subdivision of \(L_k\).

If \((z_1, z_2)\) is a pair of vertices in a graph \(H\) and \(\phi\) is an \(L_k\)-model in $H$ such that \(z_1 \in V(\phi((1, k)))\) and \(z_2 \in V(\phi((2, k)))\), then we say that \(\phi\) is \emph{rooted} at the pair \((z_1, z_2)\).

\section{\(2\)-connected centered colorings}
\label{sec:variant_center_colorings}

The goal of this section is to prove Theorem~\ref{thm:graphs-without-long-ladders}.
We start by establishing some basic properties of \(2\)-connected centered colorings.

\begin{lemma}\label{lem:l-m-apices2}
  For every graph \(G\) and subset \(X \subseteq V(G)\) we have
  \[\tdtwo(G - X) \ge \tdtwo(G) - |X|.\]
\end{lemma}
\begin{proof}
  Let \(\varphi\) be a \(2\)-connected centered coloring of \(G - X\) using at most \(\tdtwo(G - X)\) colors, and extend it
  to a coloring \(\varphi'\) of \(G\) using at most \(\tdtwo(G - X) + |X|\) colors by assigning new distinct colors to the vertices of \(X\).
  Consider a connected subgraph \(H\) of \(G\) which does not have a cutvertex.
  If \(H\) contains a vertex from \(X\), then the color of that vertex is unique in \(G\) and thus in \(H\).
  If \(H \subseteq G - X\), then some color is assigned to exactly one vertex of \(H\) because \(\varphi\) is a \(2\)-connected centered coloring of \(G - X\).
  Hence \(\varphi'\) is a \(2\)-connected centered coloring of \(G\) using \(\tdtwo(G - X) + |X|\) colors, and therefore \(\td_2(G) \le \td_2(G - X) + |X|\),
  so indeed \(\tdtwo(G - X) \ge \tdtwo(G) - |X|\).
\end{proof}

\begin{lemma}\label{lem:td2-separation}
  If \((A_1, A_2)\) is a separation of order \(\leq 1\) of a graph \(G\), then \[\tdtwo(G) = \max \set{\tdtwo(G[A_1]), \tdtwo(G[A_2])}.\]
\end{lemma}
\begin{proof}
  If \(\varphi\) is a \(2\)-connected centered coloring of \(G\),
  then the restrictions of \(\varphi\) to \(A_1\) and \(A_2\) are \(2\)-connected centered colorings of \(G[A_1]\) and \(G[A_2]\), respectively, and these colorings use no more colors than \(\varphi\).
  Therefore \(\tdtwo(G[A_i]) \le \tdtwo(G)\) for \(i \in \set{1, 2}\).

  Let \(m = \max \set{\tdtwo(G[A_1]), \tdtwo(G[A_2])}\),
  and let \(\varphi_1 \colon A_1 \to \set{1, \dots, m}\) and \(\varphi_2 \colon A_2 \to \set{1, \dots, m}\) be \(2\)-connected centered colorings of \(G[A_1]\) and \(G[A_2]\), respectively.
  Since \(|A_1 \cap A_2| \le 1\), after permuting the colors in one of the colorings \(\varphi_1\) or \(\varphi_2\) we may assume that they agree on \(A_1 \cap A_2\).
  Let \(\varphi \colon V(G) \to \set{1, \dots, m}\) be a vertex coloring such that its restriction to \(A_i\) is \(\varphi_i\) for \(i \in \set{1, 2}\).
  For every connected subgraph \(H\) of \(G\) which does not have a cutvertex, the separation \((A_1 \cap V(H), A_2 \cap V(H))\) of \(H\) must be trivial, so \(V(H) \subseteq A_i\) for some \(i \in \set{1, 2}\).
  Since the restriction of \(\varphi\) to \(A_i\) is a \(2\)-connected centered coloring of \(G[A_i]\), some vertex of \(H\) receives a color not used for any other vertex of \(H\).
  This proves that \(\varphi\) is a \(2\)-connected centered coloring of \(G\) and \(\td_2(G) \le m\) as required.
\end{proof}

\begin{lemma}\label{lem:2-cc-blocks}
  Every graph \(G\) contains a connected subgraph \(B\) without a cutvertex such that \(\tdtwo(B) = \tdtwo(G)\).
\end{lemma}
\begin{proof}
  We prove the lemma by induction on the size of the graph.
  If \(G\) is connected and does not have a cutvertex, then the lemma holds with \(B = G\).
  Otherwise, let \((A_1, A_2)\) be a nontrivial separation of order \(\leq 1\) of \(G\).
  Without loss of generality we assume that \(\tdtwo(G[A_1]) \le \tdtwo(G[A_2])\).
  By Lemma~\ref{lem:td2-separation}, we have \(\tdtwo(G[A_2]) = \tdtwo(G)\).
  By induction hypothesis applied to \(G[A_2]\) (which is a proper subgraph of \(G\)), there is a connected subgraph \(B\) of \(G[A_2]\) without a cutvertex such that \(\tdtwo(B) = \tdtwo(G[A_2]) = \tdtwo(G)\), which completes the proof.
\end{proof}

We will also need the following classical result by Erd\H{o}s and Szekeres~\cite{ES35}.

\begin{theorem}[Erd\H{o}s-Szekeres Theorem]\label{thm:erdos-szekeres}
  Let \(k \ge 1\) be an integer, let \(n = {(k-1)}^2 + 1\), and let \(a_1, \dots, a_n\) be a sequence of distinct integers.
  Then there exist integers \(i_1, \dots, i_k\) with \(1 \le i_1 < \dots < i_k \le n\) such that \(a_{i_1} < \dots < a_{i_k}\) or \(a_{i_1} > \dots > a_{i_k}\).
\end{theorem}

Recall that \(L_k\) denotes the ladder with \(k\) rungs, with vertex set \(\{1, 2\} \times \{1, \dots, k\}\),
and an \(L_k\)-model \(\phi\) is rooted at a pair \((z_1, z_2)\) if \(z_i \in V(\phi((i, k)))\) for \(i \in \set{1, 2}\).

\begin{lemma}\label{lem:lk-or-rooted-lt}
  Let \(k \ge 1\) and \(t \ge 1\) be integers, let \(s = {(k-1)}^2+2\), let \(G\) be a $2$-connected graph, and let \(x_1\) and \(x_2\) be distinct vertices of \(G\).
  If \(\tdtwo(G) > t \cdot s\), then at least one of the following holds:
  \begin{enumerate}
    \item\label{itm:lk minor} \(G\) has an \(L_k\) minor, or
    \item\label{itm:rooted-lt-model} \(G\) has an \(L_t\)-model rooted at the pair \((x_1, x_2)\).
  \end{enumerate}
\end{lemma}

\begin{proof}
  We prove the lemma by induction on \(t\).
  Suppose first \(t = 1\).  Since $G$ is connected, there is an \text{\(x_1\)--\(x_2\)} path in $G$, and so~\ref{itm:rooted-lt-model} holds.

  Now suppose that \(t \ge 2\).
  By \(2\)-connectivity of \(G\), there exist two internally disjoint \text{\(x_1\)--\(x_2\) paths} \(P\) and \(P'\).
  By Menger's Theorem, either there exist \(s + 1\) disjoint \(V(P)\)--\(V(P')\) paths in \(G\), or there exists a set of at most \(s\) vertices separating \(V(P)\) from \(V(P')\).

  \textbf{Case 1: there exist \(s + 1\) disjoint \(V(P)\)--\(V(P')\) paths \(Q_1, \dots, Q_{s + 1}\).}
  We assume that the paths are listed in the order in which they intersect the path \(P\) when traversing it from \(x_1\) to \(x_2\).
  Let \(\pi\) be a permutation of \(\{1, \dots, s + 1\}\) such that \(Q_{\pi(1)}, \dots, Q_{\pi(s+1)}\) is the order in which the paths intersect \(P'\) when traversing it from \(x_1\) to \(x_2\).
  The paths \(P\) and \(P'\) are internally disjoint, so for \(2 \le i \le s\), the path \(Q_i\) has two distinct endpoints.
  Consider the sequence \(\pi(2), \dots, \pi(s)\) of length \(s-1 = {(k-1)}^2 + 1\).
  By Theorem~\ref{thm:erdos-szekeres} applied to that sequence, there exist indices \(2 \le i_1 < \cdots < i_k \le s\) such that either \(\pi(i_1) < \cdots < \pi(i_k)\), or \(\pi(i_1) > \cdots > \pi(i_{k})\).
  In both cases \(G\) has a subgraph isomorphic to a subdivision of \(L_k\) obtained as a union of a subpath of \(P\), a subpath of \(P'\) and the paths \(Q_{i_1}\), \dots, \(Q_{i_k}\).
  Hence~\ref{itm:lk minor} is satisfied.
  See~Figure~\ref{fig:case-1}.

\begin{figure}[!h]
    \centering
    \includegraphics[scale=1.0]{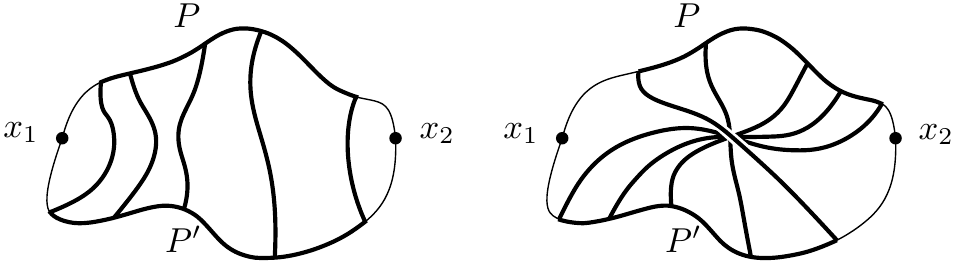}
    \caption{\label{fig:case-1}
    Case 1: two ways how the paths $P$, $P'$ and $Q_{i_1},\ldots,Q_{i_k}$ induce a subdivision of $L_k$.}
\end{figure}

  \textbf{Case 2: there is a set \(X\) of at most \(s\) vertices separating \(V(P)\) from \(V(P')\).}
  See Figure~\ref{fig:case-2}.
  We have \(\tdtwo(G) > t \cdot s\), so by Lemma~\ref{lem:l-m-apices2}, we also have \(\tdtwo(G - X) \ge \tdtwo(G) - |X| > ts - |X| \ge ts - s\).
  Hence by Lemma~\ref{lem:2-cc-blocks} there exists a connected subgraph \(B\) of \(G - X\) without a cutvertex such that \(\tdtwo(B) > ts - s\).
  Since \(\tdtwo(B) > ts - s \geq 2s - s \ge 2\), we have \(|V(B)| > 2\).
  By \(2\)-connectivity of \(G\), there exist two disjoint \(\{x_1, x_2\}\)--\(V(B)\) paths \(Q_1\) and \(Q_2\) in \(G\) with \(x_i \in V(Q_i)\) for \(i \in \{1, 2\}\).
  Let \(y_i\) denote the endpoint of \(Q_i\) in \(B\) for \(i \in \{1, 2\}\).
  Since \(\tdtwo(B) > ts - s = (t-1)s\),
  we can apply induction hypothesis to \(k\), \(t-1\), \(B\), \(y_1\) and \(y_2\).
  If \(B\) has an \(L_k\) minor, then so does \(G\), so~\ref{itm:lk minor} holds.
  Otherwise, there is an \(L_{t-1}\)-model \(\phi'\) in \(B\) such that \(y_i\) is in \(\phi'((i,t-1))\) for \(i \in \{1, 2\}\).
  Since \(B \subseteq G - X\) and the set \(X\) separates \(V(P)\) from \(V(P')\), the model \(\phi'\) intersects at most one of the paths \(P\) and \(P'\).
  Without loss of generality, let us assume that \(\phi'\) does not intersect \(P\).
  For each \(i \in \set{1, 2}\), we have \(x_i \in V(P) \cap V(Q_i)\), so \(V(P) \cap V(Q_i) \neq \emptyset\).
  Let \(R\) be a \(V(Q_1)\)--\(V(Q_2)\) subpath of \(P\).
  Since \(R\) is disjoint from the model \(\phi'\), we can see that~\ref{itm:rooted-lt-model} is witnessed by an \(L_t\)-model \(\phi\) such that
  \(\phi((i, j)) = \phi'((i, j))\) for \((i, j) \in \{1, 2\} \times \{1, \dots, t-1\}\), \(\phi((1, t)) = Q_1 - y_1\) and \(\phi((2, t)) = (Q_2 - y_2) \cup (R -  V(Q_1))\).
\begin{figure}[!h]
    \centering
    \includegraphics[scale=1.0]{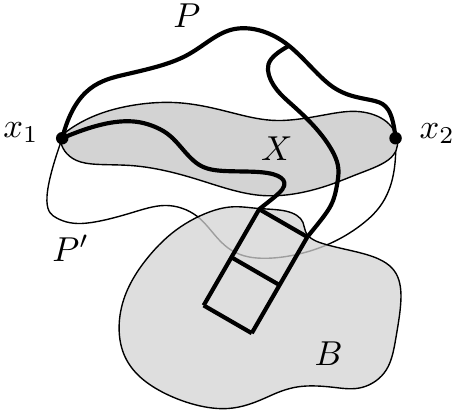}
    \caption{\label{fig:case-2}
    Case 2: An $L_3$-model in $B$ extended to an $L_4$-model rooted at $(x_1,x_2)$ in $G$.}
\end{figure}
\end{proof}

Let \(k \geq 1\) and let \(H\) be a graph isomorphic to a subdivision of \(L_k\).
Then the graph \(H\) is a union of \(k+2\) paths \(P_1\), \(P_2\), \(Q_1\), \ldots, \(Q_k\), where
\begin{enumerate}
\item the paths \(P_1\) and \(P_2\) are disjoint,
\item \(Q_1\), \dots, \(Q_k\) are disjoint \(V(P_1)\)--\(V(P_2)\) paths which intersect the paths \(P_1\) and \(P_2\) in the order in which they are listed,
\item the
paths \(Q_1\) and \(Q_k\) have their endpoints in the endpoints of the paths \(P_1\) and \(P_2\).
\end{enumerate}
For such paths \(P_1\), \(P_2\), \(Q_1\), \dots, \(Q_k\), we call the tuple \((P_1, P_2; Q_1, \dots, Q_k)\) a \emph{subdivision model} of \(L_k\) in \(H\).

\begin{lemma}\label{lem:uniform-model}
  Let \(c \geq 1\) be an integer, let \(H\) be a graph isomorphic to a subdivision of \(L_{2^c}\) with a subdivision model \((P_1, P_2; Q_1, \dots, Q_{2^c})\), and let \(\varphi\) be a cycle centered coloring of \(H\).
  If the sets of colors used by \(\varphi\) on the paths \(Q_1, \dots, Q_{2^c}\) are all the same, then \(\varphi\) uses more than \(c\) colors.
\end{lemma}
\begin{proof}
  We prove the lemma by induction on \(c\).
  It clearly holds true for \(c = 1\), so suppose that \(c \ge 2\) and the coloring \(\varphi\) on each path \(Q_i\) uses exactly the same set of colors, say \(A\).
  Since \(\varphi\) is a cycle centered coloring, there exists a vertex of unique color on the cycle \(P_1 \cup P_2 \cup Q_1 \cup Q_{2^c}\).
  Let \(x\) be such a vertex.
  We have \(\varphi(x) \not \in A\), because every color in \(A\) appears on both \(Q_1\) and \(Q_{2^c}\).
  Hence \(x\) lies either on \(P_1\) or on \(P_2\), and its color is unique in the whole graph \(H\).
  After possibly exchanging \(P_1\) and \(P_2\), we may assume that \(x \in V(P_1)\), and possibly reversing the order of the paths \(Q_1\), \dots, \(Q_{2^{c}}\), we may assume that \(x\) does not lie on the \(V(Q_1)\)--\(V(Q_{2^{c-1}})\) subpath of \(P_1\).
  The graph \(H\) contains a subgraph \(H_1\) isomorphic to a subdivision of \(L_{2^{c-1}}\) with subdivision model \((P_1', P_2'; Q_1, \dots, Q_{2^{c-1}})\), where \(P_1'\) is the \(V(Q_1)\)--\(V(Q_{2^{c-1}})\) subpath of \(P_1\) and \(P_2'\) is the \(V(Q_1)\)--\(V(Q_{2^{c-1}})\) subpath of \(P_2\).
  By induction hypothesis, \(\varphi\) uses more than \(c - 1\) colors on \(H_1\), and all these colors are distinct from \(\varphi(x)\).
  Hence \(\varphi\) uses more than \(c\) colors on \(H\).
\end{proof}

We may now prove Theorem~\ref{thm:graphs-without-long-ladders}, which we restate for convenience.

\TheoremGraphsWithoutLongLadders*

\begin{proof}
Let us first show the implication~\ref{itm:gwll-no-ladder}~\(\Rightarrow\)~\ref{itm:gwll-2-c-coloring}.
We prove the contrapositive.
Let \(\calC\) be a class of graphs such that for every integer \(m\geq 1\) there exists a graph \(G\) in \(\calC\) with \(\tdtwo(G) > m\).
We need to show that for every \(k \ge 1\) there exists a graph in \(\calC\) with an \(L_k\) minor.
Fix \(k \ge 1\) and let \(m = k({(k-1)}^2+2) + 1\).
Let \(G\) be a graph in \(\calC\) such that  \(\tdtwo(G) \ge m\).
By Lemma~\ref{lem:2-cc-blocks}, there exists a connected subgraph \(B\) of \(G\) which does not have a cutvertex such that \(\tdtwo(B) \ge m\).
Since \(m \ge 2\), the subgraph \(B\) has at least two vertices.
Let \(x_1\) and \(x_2\) be distinct vertices of \(B\).
By Lemma~\ref{lem:lk-or-rooted-lt} with \(t = k\), \(B\) has an \(L_k\) minor, and thus so does \(G\).
The implication~\ref{itm:gwll-no-ladder}~\(\Rightarrow\)~\ref{itm:gwll-2-c-coloring} follows.

The implication~\ref{itm:gwll-2-c-coloring}~\(\Rightarrow\)~\ref{itm:gwll-circ-coloring} is straightforward, since every \(2\)-connected centered coloring is a cycle centered coloring.

It remains to show the implication~\ref{itm:gwll-circ-coloring}~\(\Rightarrow\)~\ref{itm:gwll-no-ladder}.
We prove the contrapositive.
More precisely, we show that for every integer \(c \geq 1\), every graph \(G\) containing an \(L_{4^c}\) minor satisfies \(\chicir(G) > c\).
Thus, fix an integer \(c \geq 1\) and let \(G\) be a graph with an
\(L_{4^c}\) minor.
The graph \(G\) contains a subgraph isomorphic to a subdivision of \(L_{4^c}\) with a subdivision model \((P_1, P_2; Q_1, \dots, Q_{4^c})\).

Towards a contradiction, suppose that there is a cycle centered coloring of \(G\) which uses at most \(c\) colors.
For every \(i \in \set{1, \dots, 4^c}\), let \(A_i\) denote the set of colors used on the path \(Q_i\).
By the pigeonhole principle, for \(k = 2^c\), there are indices \(i_1\), \dots, \(i_k\) with \(1 \le i_1 < \dots < i_k \le 4^c\) such that \(A_{i_1} = \dots = A_{i_k}\).
Let \(P_1'\) be the \(V(Q_{i_1})\)--\(V(Q_{i_k})\) subpath of \(P_1\) and let \(P_2'\) be the \(V(Q_{i_1})\)--\(V(Q_{i_k})\) subpath of \(P_2\).
Then \(P_1' \cup P_2' \cup Q_{i_1} \cup \cdots \cup Q_{i_k}\)  is a subgraph isomorphic to a subdivision of \(L_{2^c}\) with a subdivision model \((P_1', P_2'; Q_{i_1}, \dots, Q_{i_k})\) such that the paths \(Q_{i_j}\) all use the same set of colors.
By Lemma~\ref{lem:uniform-model}, \(\varphi\) uses more than \(c\) colors on this subgraph, contradiction.
\end{proof}

\section{An application to poset dimension}
\label{sec:posets}

In this section we show Theorem~\ref{thm:posets-tdtwo}.
In a poset \(P = (X, \le_P)\), we consider the relation \(\le_P\) as a subset of \(X^2 = X \times X\). A linear order \(\le\) on \(X\) is a \emph{linear extension} of \(P\) if \({\le_P} \subseteq {\le}\).
For a set \(S \subseteq X\), we denote by \(P[S]\) the subposet of \(P\) induced by \(S\), that is \(P[S] = (S, {\le_P} \cap {S^2})\).

The next lemma is folklore, a proof is included for completeness.

\begin{lemma}\label{lem:dim-apex}
  Let \(d \ge 1\) be an integer and let \(\calC\) be a class of graphs closed under taking subgraphs such that every poset whose cover graph is in \(\calC\) has dimension at most \(d\).
  Let  \(P= (X, \le_P)\) be a poset with cover graph \(G\) such that \(G - z \in \calC\) for some vertex \(z\) of \(G\).
  Then \(\dim(P) \le 2d\).
\end{lemma}

\begin{proof}
  Let \(U = \{x \in X : x \ge_P z\}\) and let \(D = \{x \in X : x \le_P z\}\).
  It is easy to observe that the cover graphs of \(P[X \setminus U]\) and \(P[X \setminus D]\) are (induced) subgraphs of \(G - z\), and thus are in \(\calC\).
  Hence \(\dim(P[X \setminus U]) \le d\) and \(\dim(P[X \setminus D]) \le d\).
  Let \(\le_1, \dots, \le_d\) be a realizer of \(P[X \setminus U]\), and let \(\le_{d+1}, \dots, \le_{2d}\) be a realizer of \(P[X \setminus D]\).
  Finally, let \(\le_U\) be a linear extension of \(P[U]\) and let \(\le_D\) be a linear extension of \(P[D]\).

  We construct a realizer \(\le'_1, \dots, \le'_{2d}\) of \(P\) as follows:
  \[
  {\le'_i} =
  \begin{cases}
    {\le_i} \cup {(X \setminus U) \times U} \cup {\le_U}&\text{for \(i \in \{1, \dots, d\}\),}\\
    {\le_D} \cup {D \times (X \setminus D)} \cup {\le_i}&\text{for \(i \in \{d+1, \dots, 2d\}\).}
  \end{cases}
  \]
  Now it remains to show that \(\le'_1, \dots, \le'_{2d}\) is a realizer of \(P\).
  It is straightforward to verify that each \(\le'_i\) is a linear extension of \(\le_P\), so if \(x \le_P y\), then \(x \le'_i y\) for \(i \in \{1, \dots, 2d\}\).
  It remains to show that if \(x\) and \(y\) are incomparable in \(P\), then there exists \(i \in \{1, \dots, 2d\}\) such that \(y <'_i x\).
  If \(\{x, y\} \subseteq X \setminus U\), then there exists \(i \in \{1, \dots, d\}\) such that \(y <_i x\) and thus \(y <'_i x\).
  Similarly, if \(\{x, y\} \subseteq X \setminus D\), then there exists \(i \in \{d+1, \dots, 2d\}\) such that \(y <_i x\), so \(y <'_{i} x\).
  Hence we are left with the case when \(\{x, y\} \cap U \neq \emptyset\) and \(\{x, y\} \cap D \neq \emptyset\).
  Since \(D \times U \subseteq {\le_{P}}\) and the elements \(x\) and \(y\) are incomparable in \(P\), this implies one of the elements \(x\) or \(y\) is equal to \(z\).
  If \(x = z\), then \(y \not \in U\), so  \((y, x) \in (X \setminus U) \times U \subseteq <'_{1}\).
  Similarly, if \(y = z\), then \(x \not \in D\), so \((y, x) \in D \times (X \setminus D) \subseteq <'_{d+1}\).
  Therefore \(\le'_1, \dots, \le'_{2d}\) is indeed a realizer of \(P\).
\end{proof}

We will also need the following theorem.
Recall that a \emph{block} of a graph \(G\) is a maximal connected subgraph of \(G\) without a cutvertex.
The blocks can be of three types: maximal \(2\)-connected subgraphs, cut edges together with their endpoints, and isolated vertices.
Two blocks have at most one vertex in common, and such a vertex is always a cutvertex.

\begin{theorem}[Trotter, Walczak and Wang~\cite{trotter-walczak-wang}]
\label{thm:dim-cutvertices}
  Let \(d \ge 1\) be an integer and let \(\calC\) be a class of graphs such that every poset whose cover graph is in \(\calC\) has dimension at most \(d\).
  If \(P\) is a poset such that all blocks of its cover graph are in \(\calC\), then \(\dim(P) \le d+2\).
\end{theorem}

We are now ready to prove Theorem~\ref{thm:posets-tdtwo}, which we restate first.

\TheoremPosetsTDTwo*

\begin{proof}
Let us show the following slightly stronger statement, which will help the induction go through:
If the cover graph \(G\) of a poset \(P\) satisfies \(\tdtwo(G) \le m\), then \(\dim(P) \le 2^{m+1} - 2\); furthermore, \(\dim(P) \le 2^{m+1} - 4\) if $G$ is $2$-connected.
We prove the statement by induction on $m$.

For the base case ($m=1$), $G$ has no edges, and thus $P$ is an \emph{antichain},
that is, a poset in which no pair of distinct elements is comparable.
Hence $\dim(P) \leq 2$;
indeed, if the elements of \(P\) are \(x_1\), \ldots, \(x_n\), then
\(P\) has a realizer \(\{{\le}, {\le'}\}\) where
\(x_1 < \cdots < x_n\) and \(x_n <' \cdots <' x_1\).
As \(2^{m+1}-2 = 2\), the statement holds.
(Note that in this case $G$ cannot be $2$-connected, so the second part of the statement holds vacuously.)

For the inductive case ($m \geq 2$), we first establish the case where $G$ is $2$-connected.
Consider a \(2\)-connected centered coloring of $G$ with $m$ colors.
There is a vertex $z$ of $G$ whose color is unique in this coloring.
Thus \(\tdtwo(G-z) \le m-1\).
By induction and Lemma~\ref{lem:dim-apex}, we deduce that $\dim(P) \leq 2 \cdot (2^{m-1 +1} - 2)= 2^{m+1} - 4$, as desired.

Now we turn to the case that $G$ is not $2$-connected.
Then each block of $G$ is either $2$-connected, or isomorphic to $K_1$ or $K_2$.
Using that our claim holds in the $2$-connected case, and the obvious fact that a poset whose cover graph is isomorphic to $K_1$ or $K_2$ has dimension $1\leq 2^{m+1} - 4$, we deduce from Theorem~\ref{thm:dim-cutvertices} that \(\dim(P) \le (2^{m+1} - 4) + 2 = 2^{m+1} - 2 \).
\end{proof}

\section{Bumping a ladder}
\label{sec:bumping_ladders}

In this section, we prove Theorem~\ref{thm:bumping-a-ladder}.

In a graph \(G\), if \(A\) is the set of cutvertices,
and \(\zB\) is the set of blocks of \(G\), then the \emph{block graph} of \(G\) is the bipartite graph on \(A \cup \zB\), where \(a \in A\) is adjacent to \(B \in \zB\) if \(a \in V(B)\).
The block graph of a graph is always a forest.

\begin{lemma}\label{lem:many-blocks}
  Let \(m\) and \(p\) be positive integers.
  Let \(G\) be a graph with at least \(p^m\) vertices and with \(\tdtwo(G) \le m\).
  Then there exists a set \(Z \subseteq V(G)\) with \(|Z| \le m - 1\) such that \(G - Z\) has at least \(p\) blocks.
\end{lemma}
\begin{proof}
  We prove the lemma by induction on \(m\).
  If \(m = 1\), then the lemma works with \(Z = \emptyset\):
  since \(\tdtwo(G) \le 1\),
  every vertex in $G$ is an isolated vertex forming a block.

  Now suppose that \(m \ge 2\).
  Fix a \(2\)-connected centered coloring of \(G\) using at most \(m\) colors.
  If \(G\) has at least \(p\) blocks, then the lemma holds with \(Z = \emptyset\).
  Let us suppose that \(G\) has less than \(p\) blocks.
  Then some block of \(G\) has at least \(p^{m-1}\) vertices.
  Let \(B_0\) be such a block, and let \(x\) be a vertex of unique color in \(B_0\).
  Thus, \(\tdtwo(B_0 - x) \le m - 1\).
  By induction hypothesis, there exists a vertex subset \(Y\) in \(B_0 - x\) with \(|Y| \le m - 2\) such that \((B_0 - x) - Y\) has at least \(p\) blocks.
  Let \(Z = Y \cup \set{x}\).
  It remains to show that \(G - Z\) has at least \(p\) blocks.
  The graph \((B_0 - x) - Y = B_0 - Z\) is a subgraph of \(G - Z\), so every block of \(B_0 - Z\) is contained in a block of \(G - Z\).
  Since \(B_0 - Z\) has at least \(p\) blocks, it suffices to show that no block of \(G - Z\) contains two blocks of \(B_0 - Z\).

  Towards a contradiction, suppose that a block \(B\) of \(G - Z\) contains two distinct blocks \(B_1\) and \(B_2\) of \(B_0 - Z\).
  In particular, we have \(|V(B) \cap V(B_0)| \ge |V(B_1) \cup V(B_2)| \ge 2\).
  Since two blocks of \(G\) have at most one vertex in common, this implies that \(B_0\) is the (unique) block of \(G\) containing the block \(B\) of \(G - Z\).
  But \(B \subset B_0 \cap (G - Z) = B_0 - Z\), so \(B\) is a block of \(B_0 - Z\), so \(B_1 \subseteq B\) and \(B_2 \subseteq B\) imply \(B_1 = B = B_2\), contradiction.
  This concludes the proof.
\end{proof}

\begin{observation}\label{obs:critical-contraction}
  If \(G\) is a \(3\)-connected graph on at least \(5\) vertices and \(e\) is an edge in \(G\) such that \(G/e\) is not \(3\)-connected, then \(G\) admits a nontrivial separation \((A_1, A_2)\) of order \(3\) such that both endpoints of \(e\) lie in \(A_1 \cap A_2\).
\end{observation}

The following lemma is due to Halin~\cite{H69a}.
Since we are not aware of a published proof in English\footnote{We note however that the lemma appears as an exercise in Diestel's textbook~\cite[Chapter 3, exercise 10]{Diestel5thEdition}.}, we include one for the reader's convenience.

\begin{lemma}[Halin~\cite{H69a}]\label{lem:preserving3conn}
Let \(G\) be a 3-connected graph.  Let \(e \in E(G)\). If neither \(G/e\) nor \(G - e\) is 3-connected, then some endpoint of \(e\) has degree \(3\).
\end{lemma}

\begin{proof}
Suppose that \(G/e\) and \(G - e\) are not $3$-connected.
As a \(3\)-connected graph, \(G\) has at least \(4\) vertices.
If \(|V(G)| = 4\), then \(G\) is complete and the lemma holds as
all vertices are of degree \(3\).
Hence we assume that \(|V(G)| \ge 5\).

Since \(G - e\) is not \(3\)-connected and has at least \(5\)
vertices, we can fix a nontrivial \(2\)-separation $(A_1,A_2)$ of $G-e$.
As a \(3\)-connected graph, \(G\) does not admit a nontrivial \(2\)-separation,
so neither \(A_1\) nor \(A_2\) contains both endpoints of \(e\).
Hence \(e\) has an endpoint \(v_1\) in \(A_1 \setminus A_2\) and an endpoint \(v_2\) in \(A_2 \setminus A_1\).

By Obervation~\ref{obs:critical-contraction}, we can fix
a nontrivial order-\(3\) separation \((B_1, B_2)\) of \(G\)
with \(\{v_1, v_2\} \subseteq B_1 \cap B_2\).
Let \(w\) denote the vertex of \(B_1 \cap B_2\) other than \(v_1\) and \(v_2\).
As \((A_1, A_2)\) is a separation of \(G - v_1 v_2\) and \((B_1, B_2)\) is a
separation of \(G\) with \(\set{v_1, v_2} \subseteq B_1 \cap B_2\), the pair
\((A_i \cap B_j, A_{3-i} \cup B_{3-j})\) is a separation of \(G\) for \(i, j \in \set{1, 2}\).

We claim that $(A_1 \cap A_2) \setminus B_1 \neq \emptyset$ and
\((A_1 \cap A_2) \setminus B_2 \neq \emptyset\).
Towards a contradiction, suppose that it is not the case.
After possibly swapping \(B_1\) and \(B_2\) we may assume that
\begin{equation}\label{eq:inclusion}
A_1 \cap A_2 \subseteq B_2.
\end{equation}
Consider the separations \((A_1 \cap B_1, A_2 \cup B_2)\) and
\((A_2 \cap B_1, A_1 \cup B_2)\) of \(G\).
The order of those separations is at most \(2\) as by~\eqref{eq:inclusion},
for \(i \in \set{1, 2}\) we have
\begin{align*}
(A_i \cap B_1) \cap (A_{3-i} \cup B_2)
&= (A_i \cap B_1 \cap A_{3-i}) \cup (A_i \cap B_1 \cap B_2)\\
&\subseteq (A_i \cap B_1 \cap B_2) = \set{v_i, w}
\end{align*}

Since \(G\) is \(3\)-connected, both separations \((A_1 \cap B_1, A_2 \cup B_2)\)
and \((A_2 \cap B_1, A_1 \cup B_2)\) must be trivial.
However, by nontriviality of \((B_1, B_2)\), \(B_1\) is a proper subset of
\(V(G)\), and thus \(A_1 \cap B_1\) and \(A_2 \cap B_1\) are proper subsets of \(V(G)\) as well.
Therefore it must be the case that \(A_2 \cup B_2 = A_1 \cup B_2 = V(G)\), and
thus by~\eqref{eq:inclusion}
\[V(G) = (A_2 \cup B_2) \cap (A_1 \cup B_2) = (A_1 \cap A_2) \cup B_2 = B_2,\]
contradicting the nontriviality of \((B_1, B_2)\).
This contradiction proves our claim that $(A_1 \cap A_2) \setminus B_1 \neq \emptyset$ and
\((A_1 \cap A_2) \setminus B_2 \neq \emptyset\).

The order of the separation \((A_1, A_2)\) is at most \(2\), so the fact that
$(A_1 \cap A_2) \setminus B_1 \neq \emptyset$ and
\((A_1 \cap A_2) \setminus B_2 \neq \emptyset\)
implies that \(A_1 \cap A_2 = \set{y_1, y_2}\) for some two vertices
\(y_1 \in B_1 \setminus B_2\) and \(y_2 \in B_2 \setminus B_1\).

Since \(w \not \in A_1 \cap A_2\), either \(w \in A_1 \setminus A_2\) or \(w \in A_2 \setminus A_1\).
The two cases are symmetric, so let us assume that \(w \in A_2 \setminus A_1\).
Summarizing, we have \(v_i \in B_1 \cap B_2 \cap A_i \setminus A_{3-i}\) and
\(y_i \in A_1 \cap A_2 \cap B_i \setminus B_{3-i}\) for \(i \in \set{1, 2}\),
\(w \in A_2 \setminus A_1\), and \(A_1 \cap A_2 \cap B_1 \cap B_2\) is empty.
Hence \((A_1 \cap B_1, A_2 \cup B_2)\)
and \((A_1 \cap B_2, A_2 \cup B_1)\) are separations of \(G\) with
\[(A_1 \cap B_1) \cap (A_2 \cup B_2) = \set{v_1, y_1}\quad\text{and}\quad
(A_1 \cap B_2) \cap (A_2 \cup B_1) = \set{v_1, y_2}.\]
The \(3\)-connectedness of \(G\) implies that the two separations are trivial
and thus \(A_1 = \set{v_1, y_1, y_2}\).
As \((A_1, A_2)\) is a separation of \(G - e\), this implies that \(v_1\) can be
adjacent only to the vertices \(y_1\), \(y_2\) and \(v_2\), so the degree of
\(v_1\) is at most \(3\).
Since \(G\) is \(3\)-connected, the degree of \(v_1\) is at least \(3\), so
the vertex \(z = v_1\) satisfies the lemma.
\end{proof}

The next result demonstrates the usefulness of vertices of degree three when attempting to preserve 3-connectivity while taking minors.  The result is also due to Halin~\cite{H69b}; see \cite{AES87} for an alternate proof in English.

\begin{lemma}[Halin~\cite{H69b}]
\label{lem:Halin}
Let \( G \) be a 3-connected graph on at least \(5\) vertices and let \( v \in V(G) \) be a vertex of degree \(3\).  Then there exists an edge \( e\) incident to \( v\) such that \( G/e \) is 3-connected.
\end{lemma}

\begin{lemma}\label{lem:general-fan}
  Let \(G\) be a \(3\)-connected graph, let \(P\) be an induced path in \(G\), and let \(Z \subseteq V(G) \setminus V(P)\) be such that every internal vertex of \(P\) has all its neighbors in \(V(P) \cup Z\).
  If the length of \(P\) is at least \(2|Z| + 3\), then there exists an edge \(e \in E(P)\) such that \(G/e\) is \(3\)-connected.
\end{lemma}

\begin{proof}
By the \(3\)-connectivity of $G$, every internal vertex of $P$ has a neighbor in $Z$.
Hence, when we fix $k=|Z|$ and $\ell$ to be the length of $P$,
we have $k \geq 1$ and $\ell \geq 2k+3 \geq 5$.

For each \(v \in V(G)\), let \(N(v)\) denote the set of neighbors of \(v\) in \(G\).
Let $v_0, v_1, \ldots, v_\ell$ be the vertices of $P$ in the order in which they appear on the path.
For each $i \in \{1, \dots, \ell-1\}$,
let $X_i = \bigcup_{j<i} N(v_{j}) \cap Z$ and let $Y_i = \bigcup_{j>i} N(v_{j}) \cap Z$.
Since $X_1\subseteq \cdots \subseteq X_{\ell-1}$ and $|X_{\ell-1}| \le |Z| = k$,
there exist at most $k$ indices $i$ with $1 \le i \le \ell-2$ such that $X_i \neq X_{i+1}$. Symmetrically, there are at most $k$ indices $i$ with $1 \le i \le \ell-2$ such that $Y_i \neq Y_{i+1}$. Since $\ell-2 > 2k$, there exists $i$ such that $1 \le i \le \ell-2$ and $(X_i, Y_i) = (X_{i+1}, Y_{i+1})$.
So $N(v_i) \cap Z \subseteq N(\{v_0, \dots, v_{i-1}\})$ and $N(v_{i+1}) \cap Z \subseteq N(\{v_{i+2}, \dots, v_{\ell}\})$.

Let $e$ be the edge $v_iv_{i+1}$ of $P$. We claim $G/e$ is 3-connected.
Arguing by contradiction, suppose that $G/e$ is not $3$-connected.
By Observation~\ref{obs:critical-contraction} there is a nontrivial separation $(A_1, A_2)$ of $G$ of order $3$
with $v_i,v_{i+1} \in A_1 \cap A_2$.
Let $u\notin \{v_i,v_{i+1}\}$ denote the third vertex of $A_1 \cap A_2$.
Clearly, we can assume that $u$ does not lie on $v_0Pv_{i-1}$ or that $u$ does not lie on $v_{i+2}Pv_{\ell}$.
We will continue the proof assuming the former; the other case has a symmetric proof.

Lest $(A_1 \setminus \set{v_i}, A_2)$ form a \(2\)-separation in $G$, there exists a vertex $a \in A_1 \setminus A_2$ such that $a$ is a neighbor of $v_i$ in $G$.  Since $v_i$ is an internal vertex of $P$ we have $a \in Z$ or $a = v_{i-1}$.
If $a \in Z$, then since $a \in X_{i+1} = X_i$, there exists an
edge of $G$ joining $a$ to a vertex $v_{j_1}$ with $j_1 < i$, and since \(a \in A_1 \setminus A_2\), we have \(v_{j_1} \in A_1\).
If \(a = v_{i-1}\), then for \(j_1 = i-1\) the vertex \(v_{j_1}\)
lies in \(A_1\).
Thus, in either case, there exists $j_1 < i$ such that $v_{j_1} \in A_1$.

Swapping \(A_1\) and \(A_2\) in the above reasoning, there must be a vertex \(v_{j_2}\) with $j_2 < i$,
such that $v_{j_2} \in A_2$. Therefore $v_{j_1} P v_{j_2}$ intersects \(A_1 \cap A_2\), which contradicts that
\(v_0 P v_{i-1}\) has no vertices in \(A_1 \cap A_2\).  We conclude that $G/e$ is 3-connected, as claimed.
\end{proof}

The following observation will be used in the proof of Lemma~\ref{lem:rooted-ladder-from-ladder}.

\begin{observation}\label{obs:ladder-model-in-subdivision}
Let \(k\) be an integer with \(k \ge 2\), let \(H\) be a graph isomorphic to a subdivision of \(L_k\), and let \(\phi\) be an \(L_k\)-model in \(H\).
For any two distinct vertices \(x\) and \(y\) of \(L_k\),
\begin{enumerate}
\item if \(x\) and \(y\) are adjacent in \(L_k\), then there is exactly one edge between \(\phi(x)\) and \(\phi(y)\) in \(H\), and
\item if \(x\) and \(y\) are nonadjacent in \(L_k\), then there are no edges between \(\phi(x)\) and \(\phi(y)\) in \(H\).
\end{enumerate}
\end{observation}
\begin{proof}
  If this is not true, then there exists an edge \(e\) in \(H\) such that \(\phi\) is an \(L_k\)-model in \(H - e\).
  Hence there exists a proper subgraph of \(H\) which is isomorphic to a subdivision of \(L_k\), which is impossible since \(L_k\) has no vertices of degree \(1\).
\end{proof}

\begin{lemma}\label{lem:rooted-ladder-from-ladder}
  Let \(k\) be an integer with \(k \ge 2\), let \(H\) be a graph isomorphic to a subdivision of \(L_k\), and
  let \(\phi\) be an \(L_k\)-model in \(H\).
  If  \(z_1\) and \(z_2\) are distinct vertices of \(H\) with
  \(z_1 \in V(\phi((i_1, j_1)))\), \(z_2 \in V(\phi((i_2, j_2)))\) and \(1 \le j_1 \le j_2 \le k\), then there exists an \(L_{j_1}\)-model rooted at \((z_1, z_2)\) and an \(L_{k-j_2+1}\)-model rooted at \((z_1, z_2)\).
\end{lemma}

\begin{proof}
We only show the existence of an \(L_{j_1}\)-model rooted at \((z_1, z_2)\) since an \(L_{k-j_2+1}\)-model rooted at \((z_1, z_2)\) can be obtained in
a similar way using a symmetry of the ladder.

If \(j_1 = 1\), the existence of the desired model follows from connectedness of \(H\).
Let us hence assume that \(j_1 \ge 2\).
Let \(A = \bigcup_{1 \le j < j_1} V(\phi((1, j))) \cup V(\phi((2, j)))\).
As \(H\) is \(2\)-connected, we can fix two disjoint \(A\)--\(\set{z_1, z_2}\)
paths \(R_1\) and \(R_2\).
By Observation~\ref{obs:ladder-model-in-subdivision}, in \(H\) there are exactly two edges between \(A\) and \(V(H) \setminus A\), and thus, since \(z_1\) and \(z_2\) do not lie in \(A\), one of the
paths \(R_1\) and \(R_2\) has an endpoint in \(V(\phi((1, j_1-1)))\),
and the other in \(V(\phi((2, j_1-1)))\).
We assume that \(R_1\) has an endpoint \(v_1\) in \(V(\phi((1, j_1-1)))\) and \(R_2\)
has an endpoint \(v_2\) in \(V(\phi((2, j_1-1)))\).
Since \(H - A\) is connected, we can fix a \(V(R_1)\)--\(V(R_2)\) path \(S\)
in \(H - A\).
Consider the \(L_{j_1}\)-model \(\phi'\) defined as follows.
\[
\phi'((i, j))=
\begin{cases}
\phi((i, j)) &\text{if }j < j_1,\\
R_1 - v_1 &\text{if } (i, j) = (1, {j_1}),\\
(R_2 - v_2) \cup (S - V(R_1)) &\text{if } (i, j) = (2, {j_1}).
\end{cases}
\]
The model \(\phi'\) is rooted at \((z_1, z_2)\) or \((z_2, z_1)\), so after
possibly swapping all \(\phi'((1, j))\) with \(\phi'((2, j))\) we obtain the
desired model.
\end{proof}

\begin{lemma}\label{lem:rooted-half-ladder}
  Let \(k\) be an integer with \(k \ge 2\) and let \(H\) be a graph isomorphic to a subdivision of \(L_k\).
  For every \(3\)-element vertex subset \(Z\) in \(H\) there exists a pair \((z_1, z_2)\) of distinct vertices in \(Z\) and an \(L_{\ceil{(k+1)/2}}\)-model in \(H\) rooted at \((z_1, z_2)\).
\end{lemma}
\begin{proof}
Fix an \(L_{k}\)-model \(\phi\) in \(H\).
Let \(Z = \set{z_1, z_2, z_3}\),
and let \((i_1, j_1)\), \((i_2, j_2)\) and \((i_3, j_3)\) be three vertices of \(L_k\) such that
\(z_a \in V(\phi((i_a, j_a)))\) for \(a \in \set{1, 2, 3}\).
Without loss of generality we assume that \(j_1 \le j_2 \le j_3\).
By Lemma~\ref{lem:rooted-ladder-from-ladder},
there exist an \(L_{k-j_2+1}\)-model rooted at \((z_1, z_2)\) and an \(L_{j_2}\)-model rooted at \((z_2, z_3)\).
One of these models contains a desired rooted model of \(L_{\ceil{(k+1)/2}}\).
\end{proof}

\begin{lemma}\label{lem:new-building-ladder}
  Let \(k\) and \(m\) be integers with \(k \ge 2\) and \(m \ge 1\),
  let \(G\) be a graph,
  and let \(H^1\), \dots, \(H^{m^2 + 1}\) be disjoint subgraphs of \(G\) each isomorphic to a subdivision of \(L_k\).
  If $G$ contains as a subgraph a forest \(F\) with at most \(m\) components, such that for every \(i \in \set{1, \dots, m^2+1}\) the graph \(H^i\) intersects each component of \(F\) in at most one vertex and \(H^i\) contains vertices from at least three components of \(F\),
  then \(G\) has an \(L_{k+1}\) minor.
\end{lemma}

\begin{proof}
  By Lemma~\ref{lem:rooted-half-ladder}, for every \(i \in \set{1, \dots, m^2 + 1}\) there exists a pair \((z^i_1, z^i_2)\) of vertices from \(V(H^i) \cap V(F)\) such that there exists an \(L_{\ceil{(k+1)/2}}\)-model \(\phi^i\) in \(H^i\) which is rooted at \((z^i_1, z^i_2)\).
  The vertices \(z^i_1\) and \(z^i_2\) lie in distinct components of \(F\), and by the pigeonhole principle, there exist distinct indices \(i\) and \(j\) in \(\set{1, \dots, m^2 + 1}\) such that
  the vertices \(z^i_1\) and \(z^j_1\) lie in the same component of \(F\), and the vertices \(z^i_2\) and \(z^j_2\) lie in the same component of \(F\).
  Let \(P_1\) be the \(z_1^i\)--\(z_1^j\) path in \(F\) and let \(P_2\) be the \(z_2^i\)--\(z_2^j\) path in \(F\).
  Gluing the \(L_{\ceil{(k+1)/2}}\)-models \(\phi^i\) and \(\phi^j\) using paths \(P_1\) and \(P_2\), we obtain a model of \(L_{k+1}\) in \(G\).
\end{proof}

We may now turn to the proof of Theorem~\ref{thm:bumping-a-ladder}, which we restate first.

\TheoremBumpingLadder*

\begin{proof}
Let us first observe that if \(k = 1\), then the theorem holds trivially because every \(3\)-connected graph contains a cycle of length at least \(4\), which is isomorphic to a subdivision of \(L_2\).
Thus, we may assume that $k \ge 2$.
Fix $m$ to be the minimum positive integer such that every graph \(G\) with no $L_{k+1}$ minor satisfies $\tdtwo(G) \le m$.
Note that $m$ is well defined by Theorem~\ref{thm:graphs-without-long-ladders}.
Let $\lambda_0$, $\lambda$, and $p$ be positive integers defined as follows.
\begin{align*}
  \lambda_0 &:= ((2m+2)(4k+1)+4)(m^2+2k^3(m-1)+1); \\
  \lambda &:= 6m(\lambda_0+3)+2;\\
  p &:= ({(m-1)}^2 + 4(m-1))(m^2 +4(m-1))(\lambda+1).
\end{align*}

We show that the theorem holds for \(N = p^m\).
Suppose not and fix $G$ to be a minor minimal 3-connected graph which contains a union of \(N\) disjoint copies of \(L_{k}\) as a minor, but which does not contain $L_{k+1}$ as a minor.
Hence we can choose disjoint subgraphs \(H_1\), \dots, \(H_N\)  of \(G\) such that each \(H_i\) is isomorphic to a subdivision of \(L_k\).
Choose such \(H_1\), \ldots, \(H_N\) which minimize the value of
\(|V(H_1 \cup \dots \cup H_N)|\).

Let \(H = H_1 \cup \dots \cup H_N\).
For each \(H_i\) fix a \(4\)-element vertex subset \(X_i\) such that the graph obtained from \(H_i\) by suppressing every vertex of degree two which is not in $X$ results in a graph isomorphic to $L_k$.
Let \(X = X_1 \cup \dots \cup X_N\).
\begin{claim}\label{cl:deg3}
Every vertex of \(G\) adjacent to a vertex not in $V(H)$ has degree \(3\) and belongs to \(X\).
\end{claim}
\begin{cproof}
Let \(v \in V(G)\) be a vertex adjacent to a vertex \(u \in V(G) \setminus V(H)\) in \(G\).
Both $G/uv$ and $G-uv$ contain $H$ as a minor, so by minimality, neither $G/uv$ nor $G-uv$ is $3$-connected.  By Lemmas~\ref{lem:preserving3conn} and \ref{lem:Halin},  there exists $x \in \{u,v\}$ of degree $3$ in \(G\) and an edge $e$ incident to $x$ such that $G/e$ is 3-connected.
If some endpoint of $e$ is not in $V(H)$ then again $G/e$ contains $H$ as a minor, and we have a contradiction to minimality.
Thus, both endpoints of $e$ are in $V(H)$, and hence $x=v$.
As $v$ has degree $3$ in $G$ and $u \in V(G) \setminus V(H)$ is its neighbor,
we conclude that $v$ has degree exactly two in $H$ and that $e\in E(H)$.
Finally, if $v \notin X$, then $H/e$ contains as a minor $N$ disjoint copies of $L_k$ and we see that $G/e$ contradicts the minimality of our counterexample $G$. Thus, $v \in X$.
\end{cproof}


Since \(G\) has no \(L_{k+1}\) minor, we have \(\tdtwo(G) \le m\).
We have \(|V(G)| \ge N = p^m\), so by Lemma~\ref{lem:many-blocks}, there exists a vertex subset \(Z\) of \(G\) with \(|Z| \le m-1\)
such that \(G - Z\) has at least \(p\) blocks.
Fix such a set \(Z\).

Let us \emph{mark} each vertex $v$ of \(G-Z\) which satisfies
\begin{enumerate}
\item[$a)$] $v \in V(H)$ and $v$ is adjacent in $H$ to a vertex from \(Z \cap V(H)\), or
\item[$b)$] $v \in V(G) \setminus V(H)$ and $v$ is adjacent in $G$ to a vertex from $X \cap Z$.
\end{enumerate}

Since \(|Z| \le m - 1\) and the maximum degree of \(H\) is at most \(3\), the number of vertices which satisfy $a)$ is at most $3(m-1)$.  As every vertex in $X$ with a neighbor in $V(G) \setminus V(H)$ has degree $3$ by Claim \ref{cl:deg3}, there are at most $m-1$ vertices which satisfy $b)$.
We deduce the following.
\begin{claim}\label{clm:marked}
 There are at most $4m - 4$ marked vertices.
\end{claim}

\begin{claim}\label{clm:few-components}
  \(G - Z\) has at most \({(m-1)}^2 + 4(m-1)\) components.
\end{claim}
\begin{cproof}
Suppose to the contrary that \(G - Z\) has \({(m-1)}^2 + 4m - 4 + 1\) distinct components
\(C^1, \dots, C^{{(m-1)}^2+4m-3}\).
By Claim \ref{clm:marked}, we can assume without loss of generality that none of the components \(C^1, \dots, C^{{(m-1)}^2 +1}\) contains a marked vertex.
For each \(i \in \set{1, \dots, {(m-1)}^2 +1}\), choose an arbitrary vertex $v$ of \(C^i\).
If $v$ is in $H$, let $H^i$ denote the (unique) subgraph among \(H_1\), \dots, \(H_N\) which contains $v$.
If $v$ is not in $H$, then, by Claim~\ref{cl:deg3}, $v$ has a neighbor $u \in X$.
As $v$ is not marked, $u\not\in Z$ and therefore $u$ lies in $C^i$ as well.
Fix such a neighbor $u$ and let $H^i$ be the component of $H$ containing the vertex $u$.
Since \(C^i\) does not contain a marked vertex, \(H^i\) must be a subgraph of \(C^i\).

For every \(i \in \set{1, \dots, {(m-1)}^2+1}\) we have $|V(C^i)| \ge 4$, since $C^i$ contains a subdivision of $L_k$.  The 3-connectivity of $G$ implies that the set $Z$ has at least three vertices.  By Menger's theorem, for every \(i \in \set{1, \dots, {(m-1)}^2 +1}\), we can fix three disjoint \(V(H^i)\)--\(Z\) paths \(Q^i_1\), \(Q^i_2\) and \(Q^i_3\) with all internal vertices contained in $V(C^i)$.
Let \(F = \bigcup_{i=1}^{(m-1)^2+1} (Q^i_1 \cup Q^i_2 \cup Q^i_3)\).
Each component of the graph \(F\) is the union of a number of paths which have a common endpoint in \(Z\) but are otherwise disjoint.
Thus \(F\) is a forest, and since \(|Z| \le m-1\), it has at most \(m-1\) components.
For every \(i \in \set{1, \dots, {(m-1)}^2+1}\) the subgraph \(H^i\) intersects three components of \(F\), each in one vertex.
By Lemma~\ref{lem:new-building-ladder}, \(G\) has an \(L_{k+1}\) minor, contradiction.
\end{cproof}

\begin{claim}\label{clm:few-leaves}
  The block graph of every component of \(G - Z\) is a tree with at most \(m^2 + 4(m-1)\) leaves.
\end{claim}
\begin{cproof}
Suppose to the contrary that there is a component \(C\) of \(G - Z\) whose block graph \(\calT\) is a tree with more than \(m^2+4(m-1)\) leaves.
Let \(B^1\), \dots, \(B^{{m}^2 + 4(m-1)+1}\) be distinct leaves of \(\calT\),
and for each \(i \in \set{1, \dots, m^2 + 4(m-1) + 1}\), let \(a^i\) be the cutvertex of \(G - Z\) adjacent to \(B^i\) in \(\calT\).
By Claim \ref{clm:marked}, we may assume without loss of generality that for \(i \in \set{1, \dots, {m}^2 + 1}\), there is no marked vertex in \(B^i - a^i\).
For each \(i \in \set{1, \dots, {m^2 + 1}}\) we show there exists a component \(H^i\) of \(H\) contained in \(B^i\).
Fix an index \(i \in \set{1, \dots, {m^2 + 1}}\) and let $v$ be a vertex of $B^i - a^i$.
If $v \in V(H)$, fix $H^i$ to be the component of $H$ which contains $v$.
If $v \in V(G) \setminus V(H)$, then $v$ has at least three neighbors all of which must be contained in $X$. Fix a neighbor \(u \in X \setminus \set{a_i}\) of \(v\). As $v$ is not marked, $u$ is a vertex of $B^i$. In this case, fix $H^i$ to be the component of $H$ which contains $u$.
Since \(B^i - a^i\) has no marked vertex, no edge of \(H^i\) connects \(B^i - a^i\) to \(Z\).
Since $a^i$ is the only cutvertex of $G-Z$ in $B^i$ and since $H_i$ is $2$-connected,
we conclude that $H^i$ is contained in $B^i$.

Note that \(B^i - a^i\) is a component of \(G - (Z \cup \set{a^i})\).
As \(G\) is \(3\)-connected, this implies that \(|Z \cup \set{a^i}| \ge 3\).
Hence by Menger's theorem, there exist three disjoint paths from $V(H^i)$ to the set $Z \cup \set{a^i}$ with all internal vertices contained in $B^i - a^i$.
Fix such three disjoint \(V(H^i)\)--\((Z \cup \set{a^i})\) paths \(Q^i_1\), \(Q^i_2\) and \(Q^i_3\).

Let \(T\) be a tree in \(G - Z\) which contains all vertices \(a^1\), \dots, \(a^{m^2+1}\) but none of the vertices of \(B^i - a^i\) for \(i \in \set{1, \dots, m^2+1}\), and let \(F = T \cup \bigcup_{i=1}^{m^2+1}(Q^i_1 \cup Q^i_2 \cup Q^i_3)\).
Each component of \(F\) is either a tree obtained as the union of paths which have a common endpoint in \(Z\) but are otherwise disjoint, or a tree obtained as the union of \(T\) and paths with one endpoint at \(a^i \in V(T)\) for some $i\in\{1,\ldots,m^2+1\}$.
Since \(|Z| \le m-1\), this implies that \(F\) is a forest with at most \(m\) components.
Moreover, for every \(i \in \set{1, \dots, {m}^2+1}\),
the graph \(H^i\) intersects each component of \(F\) in at most one vertex, and \(H^i\) has a non-empty intersection with at least three components of \(F\).
Hence by Lemma~\ref{lem:new-building-ladder}, \(G\) has an \(L_{k+1}\) minor, contradiction.
\end{cproof}

Recall that the graph \(G - Z\) has at least \(p\geq ({(m-1)}^2 + 4(m-1))(m^2+4(m-1))(\lambda+1)\) blocks.
Hence by Claim~\ref{clm:few-components}, there exists a component \(C\) of \(G - Z\) with at least \((m^2+4(m-1))(\lambda+1)\) blocks.
Let \(\calT\) be the block graph of \(C\).
By Claim~\ref{clm:few-leaves}, \(\calT\) has at most \((m^2+4(m-1))\) leaves, and thus \(\calT\) contains a path on \(\lambda+1\) blocks and \(\lambda\) cutvertices.
Let \(B_0 a_1 B_1 a_2 \dots a_{\lambda} B_{\lambda}\) be such a path.
Let \(\calT'\) denote the subgraph of \(\calT\) obtained by removing all edges of the form \(a_i B_i\) with \(i \in \set{1, \dots, \lambda}\).
For every \(i \in \set{0, \dots, \lambda}\) let \(W_i\) be the
set of all vertices contained in those blocks of \(C\) which lie in the same component of \(\calT'\) as \(B_i\).
The sets \(W_0\), \dots, \(W_\lambda\) obtained this way induce connected subgraphs of \(C\) whose union is \(C\), and for every \(i \in \set{1, \dots, \lambda}\), we have
\begin{equation*}
(W_0 \cup \dots \cup W_{i-1}) \cap (W_{i} \cup \dots \cup W_\lambda) = W_{i-1} \cap W_i = \set{a_i}.
\end{equation*}
See~Figure~\ref{fig:path-of-B-blocks}.

\begin{figure}[!h]
    \centering
    \includegraphics[scale=1.0]{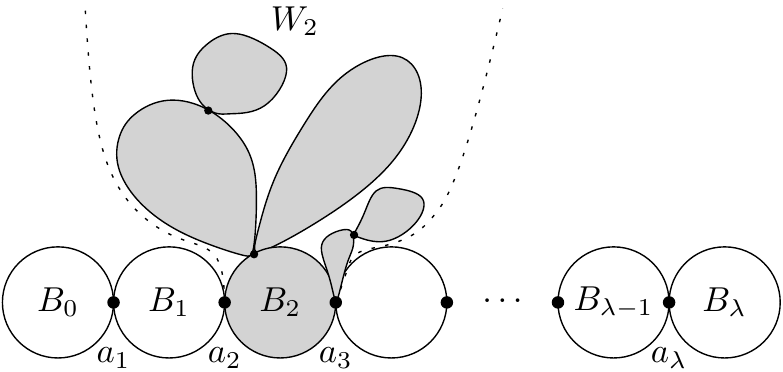}
    \caption{\label{fig:path-of-B-blocks}
    Blocks of the component $C$ of $G-Z$. The block graph $\calT$ of $C$ contains a long path \(B_0 a_1 B_1 a_2 \dots a_{\lambda} B_{\lambda}\). The remaining blocks are grouped into bundles. The bundle $W_2$ is highlighted in gray.}
\end{figure}

\begin{claim}\label{cl:clean-subpath}
  There exists an index \(i_0\) with \(1 \le i_0\) and \(i_0 + \lambda_0 \le \lambda - 1\) such that
  \begin{enumerate}
    \item\label{itm:nomarked} \(W_{i_0} \cup \dots \cup W_{i_0+\lambda_0}\) does not contain a marked vertex, and
    \item\label{itm:twopaths} there exist two disjoint \((W_0 \cup \dots \cup W_{i_0-1})\)--\((W_{i_0 + \lambda_0 +1} \cup \dots \cup W_\lambda)\) paths in $G$ which are internally disjoint from \(C\). See~Figure~\ref{fig:path-of-W-blocks}.
  \end{enumerate}
\end{claim}
\begin{figure}[!h]
    \centering
    \includegraphics[scale=1.0]{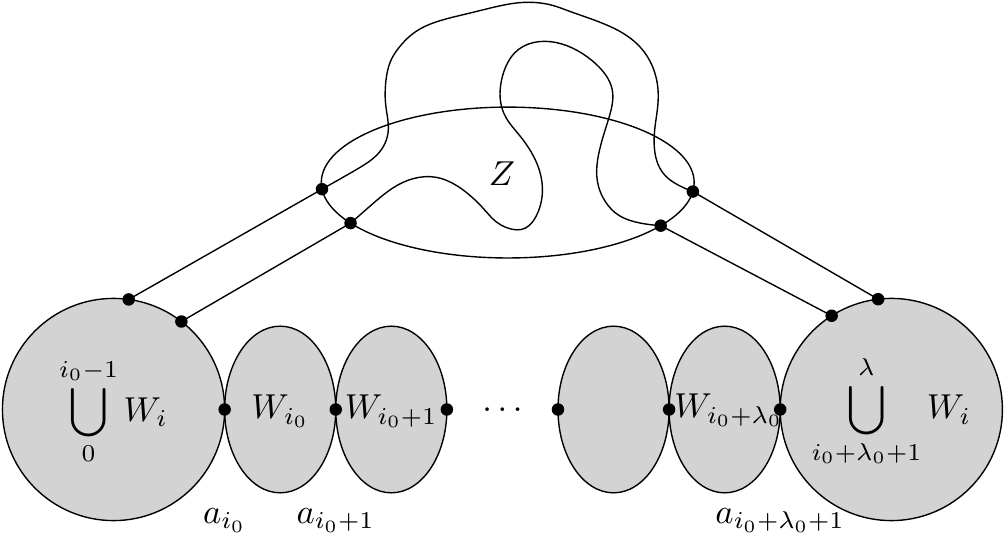}
    \caption{\label{fig:path-of-W-blocks}
    \(W_{i_0} \cup \dots \cup W_{i_0+\lambda_0}\) does not contain a marked vertex, and there are two disjoint paths connecting outside regions and internally disjoint with $C$.}
\end{figure}
\begin{cproof}
  For each \(z \in Z\), let \(J_z\) denote the set of the two largest indices
  \(i \in \set{2, \dots, \lambda - 1}\) such that \(z\) is adjacent to a vertex from \(W_i \setminus \set{a_{i}}\) (if there are less than two such indices, let \(J_z\) consist of all these indices, possibly \(J_z = \emptyset\)).
  Let \(J = \bigcup_{z \in Z} J_z\).
  As \(|Z| < m\), we have \(|J| < 2m\).

  By Claim \ref{clm:marked}, there are less than \(6m\) indices \(i \in \set{2, \dots, \lambda - 1}\) such that \(i \in J\) or \(W_i \setminus \set{a_{i+1}}\) contains a marked vertex.
  Consider the set \(\set{2, \dots, \lambda - 1}\) of \(\lambda-2\) consecutive indices.
  As \(\lambda - 2 = 6m(\lambda_0 + 3)\), by the pigeonhole principle we can find in this set at least \(\lambda_0 + 3\) consecutive indices $i$ such that \(i \not \in J\) and \(W_i \setminus \set{a_{i+1}}\) does not contain a marked vertex.
  We fix a sequence of \(\lambda_0 + 2\) such indices:
  Let \(i_0\) be an index with \(2\le i_0\) and \(i_0 + \lambda_0  + 1 \le \lambda - 1\) such that \(i \not \in J\) and \(W_i \setminus \set{a_{i+1}}\) does not contain a marked vertex, for every \(i\in \set{i_0, \dots, i_0 + \lambda_0 + 1}\).
  We have
  \[\bigcup_{i=i_0}^{i_0 + \lambda_0} W_i \subseteq \bigcup_{i=i_0}^{i_0+\lambda_0+1} W_i \setminus \set{a_{i+1}}.\]
  By the choice of \(i_0\), the right-hand side does not contain a marked vertex, therefore the index \(i_0\) satisfies \ref{itm:nomarked}.

  Partition the set \(V(C) \setminus \set{a_{i_0}}\) into the sets \(U_1\), \(U_2\) and \(U_3\), where
  \begin{align*}
      U_1 &= (W_0 \cup \dots \cup W_{i_0-1}) \setminus \set{a_{i_0}},\\
      U_2 &= (W_{i_0} \cup \dots \cup W_{i_0 + \lambda_0}) \setminus \set{a_{i_0}}\text{, and}\\
      U_3 &= (W_{i_0 + \lambda_0 + 1} \cup \dots \cup W_{\lambda}) \setminus \set{a_{i_0 + \lambda_0 + 1}}.
  \end{align*}
  Since \(G\) is \(3\)-connected, the graph \(G - a_{i_0}\) is \(2\)-connected.
  We have \(|U_1| \ge 2\) as \(i_0 \ge 2\), and \(|U_2 \cup U_3| \ge 2\) as \(\lambda - i_0 \ge  \lambda_0 \ge 2\), so there exist two disjoint \(U_1\)--\((U_2 \cup U_3)\) paths \(P_1\) and \(P_2\) in \(G - a_{i_0}\).
  For \(i \in \set{1, 2}\), let \(v_i\) denote the endpoint of \(P_i\) contained in \(U_2 \cup U_3\), and let \(z_i\) denote the vertex adjacent to \(v_i\) in \(P_i\).
  Since \(a_{i_0}\) is a cutvertex in \(C\), we deduce that \(\set{z_1, z_2} \subseteq Z\).

  Each of the paths \(P_1\) and \(P_2\) is internally disjoint from \(V(C)\) and has one endpoint in \(U_1\), but the other endpoint possibly lies in \(U_2\) and not in \(U_3\) as required by~\ref{itm:twopaths}.
  We will show however that if for some \(i \in \set{1, 2}\) we have \(v_i \in U_2\), then \(z_i\) is adjacent to at least two vertices in \(U_3\).
  This will imply that by replacing the endpoints of \(P_1\) and \(P_2\) contained in \(U_2\) we can obtain two disjoint \(U_1\)--\(U_3\) paths in \(G\) which are internally disjoint from \(V(C)\), thus proving that \(i_0\) satisfies~\ref{itm:twopaths}.

  Let us hence fix an index \(i \in \set{1, 2}\) such that \(v_i \in U_2\), and let \(j \in \set{i_0, \dots, i_0 + \lambda_0}\) be such that \(v_i \in W_j \setminus \set{a_j}\).
  By our choice of \(i_0\), we have \(j \not \in J_{z_i}\).
  Since \(z_i\) is adjacent to the vertex  \(v_i \in W_j \setminus \set{a_j}\), the definition of \(J_{z_i}\) implies that \(|J_{z_i}| = 2\) and both elements of \(J_{z_i}\) are greater than \(i_0 + \lambda_0\).
  Therefore \(z_i\) is indeed adjacent to at least two vertices in \(U_3\), completing the proof of~\ref{itm:twopaths}
\end{cproof}

Let us fix an index \(i_0\) as in Claim~\ref{cl:clean-subpath}.

\begin{claim}\label{clm:cc}
  For every index \(i \in \set{i_0, \dots, i_0 + \lambda_0 - 2}\), either
  \begin{enumerate}
      \item\label{itm:contained} some component of \(H\) has all its vertices contained in one of the sets \(W_i\), \(W_{i+1}\) or \(W_{i+2}\), or
      \item\label{itm:connected} there is a path in \(H\) between $a_i$ and $a_{i+1}$ with all vertices in $W_i$.
  \end{enumerate}
\end{claim}
\begin{cproof}
  For the contrary, suppose that for a fixed index $i\in\set{i_0, \dots, i_0 + \lambda_0 - 2}$ both items do not hold.
  By Claim~\ref{cl:deg3}, \(V(G)\setminus V(H)\) is an independent set in \(G\).
  As \(W_{i+1}\) contains two endpoints of an edge in the block \(B_{i+1}\), this implies
  that \(V(H) \cap W_{i+1} \neq \emptyset\).
  Fix a component \(H^0\) of \(H\) such that \(V(H^0) \cap W_{i+1} \neq \emptyset\).

  Since the set \(W_{i} \cup W_{i + 1} \cup W_{i + 2}\) does not contain a marked vertex, no edge of \(H^0\) connects a vertex in \(W_{i} \cup W_{i + 1} \cup W_{i + 2}\) to a vertex in \(Z\).
  As \(a_{i+1}\) and \(a_{i+2}\) are cutvertices of \(G[W_{i} \cup W_{i + 1} \cup W_{i + 2}]\) and \(H^0\) is \(2\)-connected, this implies that if \(V(H^0) \subseteq W_{i} \cup W_{i + 1} \cup W_{i + 2}\), then \ref{itm:contained} holds.
  Let us hence assume that
\[V(H^0) \setminus (W_i \cup W_{i+1} \cup W_{i+2}) \neq \emptyset.\]

  Let \(U = (V(G) \setminus (W_i \cup W_{i+1} \cup W_{i+2})) \cup \set{a_i, a_{i+3}}\).
  Note that \(a_i \in V(H^0)\) as otherwise \(((W_i \cup W_{i+1} \cup W_{i+2}) \cap V(H^0), U \cap V(H^0))\) is a nontrivial \(1\)-separation of the \(2\)-connected \(H^0\).
  Moreover, if \(a_{i+1} \not \in V(H^0)\), then \(((W_{i+1} \cup W_{i+2}) \cap V(H^0), (U \cup W_i) \cap V(H^0))\) is a nontrivial \(1\)-separation of \(H^0\), so \(a_{i+1} \in V(H^0)\).
  Therefore \(\set{a_i, a_{i+1}} \subseteq V(H^0)\).

  Towards a contradiction, suppose that \(H^0\) does not contain a path between \(a_i\) and \(a_{i+1}\) with all vertices in \(W_i\), and let \(D\) be the component of \(H^0[W_i \cap V(H^0)]\) which contains \(a_i\) but not \(a_{i+1}\).
  In such case we obtain a nontrivial \(1\)-separation
  \(((W_i \cup W_{i+1}\setminus V(D)) \cap V(H^0), (W_{i+2} \cup U \cup V(D)) \cap V(H^0))\) of \(H^0\), which is a contradiction.
  Hence~\ref{itm:connected} holds.
\end{cproof}

\begin{claim}\label{clm:contained-ladder}
  Among any \((2m+2)(4k+1)+3\) consecutive indices in
  \(\set{i_0, \dots, i_0 + \lambda_0}\) we can find an index \(i\) such that
  some component of \(H\) has all its vertices contained in \(W_i\).
\end{claim}

\begin{cproof}
  Suppose to the contrary that there is an index \(\alpha\) with \(i_0 \le \alpha\) and \(\alpha + (2m+2)(4k+1) +2\le i_0 + \lambda_0\) such that for every \(i \in \set{\alpha, \dots, \alpha + (2m+2)(4k+1)+2}\) there is no \(j \in \set{1, \dots, N}\) with \(V(H_j)\subseteq W_i\).

  By Claim~\ref{clm:cc} applied to all indices \(i \in \set{\alpha, \dots, \alpha + (2m+2)(4k+1)}\), we can fix a component \(H^0\) of \(H\) and a path \(P\) contained in \(H^0\) with endpoints
  in \(a_\alpha\) and \(a_{\alpha + (2m+2)(4k+1) + 1}\) which traverses the vertices
  \(a_\alpha, \dots, a_{\alpha + (2m+2)(4k+1) + 1}\) in that order and the subpath of \(P\)
  from \(a_i\) to \(a_{i+1}\) has all its vertices in \(W_i\) for
  \(i \in \set{\alpha, \dots, \alpha + (2m+2)(4k+1)}\).

  The graph \(H^0\) contains \(2k\) vertices which are of degree \(3\) or belong to \(X\).
  Such vertices lie in at most \(4k\) of the sets \(W_{\alpha}\), \dots, \(W_{\alpha + (2m+2)(4k+1)}\).
  Hence there exists an index \(\alpha'\) with \(\alpha \le \alpha'\) and
  \(\alpha' + (2m+1) \le \alpha + (2m+2)(4k+1)\)
  such that \(W_{\alpha'} \cup \dots \cup W_{\alpha' + (2m+1)}\) does not contain
  a vertex of degree \(3\) in \(H^0\) nor a vertex from \(V(H^0) \cap X\).
  This in particular implies that the subgraph of \(H^0\) induced by those vertices which lie
  in \(W_i\) is the subpath \(a_i P a_{i+1}\) for \(i \in \set{\alpha', \dots, \alpha' + 2m+1}\).
  Let us fix such an index \(\alpha'\).

  We claim that actually \(a_i P a_{i+1}\) contains all vertices in \(W_i\)
  for \(i \in \set{\alpha', \dots, \alpha' + 2m+1}\).
  Suppose to the contrary that for some index \(i\) there is a vertex \(v \in W_i \setminus V(a_i P a_{i+1})\).
  We consider two cases: when \(v \in V(H)\) and when \(v \not \in V(H)\).
  In the former case, the component of \(H\) containing \(v\) must contain all its vertices in \(W_i\)
  as \(W_i\) has no marked vertices and \(\set{a_i, a_{i+1}} \subseteq V(P) \subseteq V(H^0)\),
  and this is a contradiction.
  In the latter case, by Claim~\ref{cl:deg3}, every neighbor of \(v\) lies in
  \(X\).
  But \(v\) cannot have neighbors in \(X \cap Z\) because \(W_i\) does not contain
  marked vertices, and \(X \cap W_i = \emptyset\) by our choice of \(\alpha'\).
  Hence \(v\) must be an isolated vertex in \(G\), contradicting its \(3\)-connectedness.
  This completes the proof that \(a_i P a_{i+1}\) contains all vertices in \(W_i\).

  Note finally that \(a_i P a_{i+1}\) is an induced path in \(G\), as otherwise we could
  replace the component \(H^0\) of \(H\) with a subdivision of \(L_k\) whose vertex set is a proper subset of
  \(V(H^0)\) thus contradicting the minimality of $|V(H)|$.

  Hence \(P':=a_{\alpha'} P a_{\alpha'+2m+2}\) is an induced path in \(G\) of length at least $2m+2 \geq 2|Z|+3$ such that each internal vertex of $P'$ has all its neighbors in \(V(P') \cup Z\).
  By Lemma~\ref{lem:general-fan}, there is an edge \(e\) of $P'$ whose contraction preserves the \(3\)-connectivity of \(G\).
  As \(H^0/e\) is still isomorphic to a subdivision of \(L_k\), we get a contradiction to the minimality of $|V(H)|$.
  The lemma follows.
\end{cproof}

By Claim~\ref{clm:contained-ladder} and the choice of \(\lambda_0\), there exists a set of indices \(I\) with \(|I| \geq m^2 + 2k^3(m-1)+1\) with the property that for every \(i \in I\), there exists \(j \in \set{1, \dots, N}\) such that \(H_j \subseteq W_i\), and for distinct \(i, i' \in I\), the sets \(W_i\) and \(W_{i'}\) are disjoint, that is, \(|i'-i| \geq 2\).

For each \(i \in I\), let \(H^i\) be one of the graphs \(H_1\), \dots, \(H_N\) that is contained in \(W_i\).
For every \(i \in I\), fix a linkage $Q_1^i, Q_2^i, Q_3^i$ from $\{a_i, a_{i+1}\} \cup Z$ to $V(H^i)$ in which the number of paths having an endpoint in \(\set{a_i, a_{i+1}}\) is largest possible.
Note that $V(Q_j^i) \subseteq W_i \cup Z$ for $j=1,2,3$.
We can classify the linkages $Q_1^i, Q_2^i, Q_3^i$ by whether zero, one or two of the paths have and endpoint in $\{a_i, a_{i+1}\}$.
For $j \in \set{0, 1, 2}$, define $I_j \subseteq I$ to be the subset of indices $i\in I$ such that exactly $j$ of the paths $Q_1^i, Q_2^i, Q_3^i$ have an endpoint in \(\set{a_i, a_{i+1}}\).

\begin{claim}\label{cl:I0bded}
$I_0 = \emptyset$.
\end{claim}
\begin{cproof}
Suppose to the contrary that there exists \(i \in I_0\).
Let \(R\) be a path from \(\set{a_i, a_{i+1}}\) to \(V(H^i \cup Q^i_1 \cup Q^i_2 \cup Q^i_3)\) in \(W_i\).
The path \(R\) intersects at most one of the paths \(Q^i_1\), \(Q^i_2\), \(Q^i_3\), so without loss of generality we may assume that \(R\) does not intersect \(Q^i_1 \cup Q^i_2\).
In \(Q^i_3 \cup R\) we can find a path \(Q'\) from \(\set{a_i, a_{i+1}}\) to \(H^i\).
The path \(Q'\) does not intersect \(Q^i_1\) nor \(Q^i_2\), so after replacing \(Q^i_3\) with \(Q'\) we obtain a linkage with one more path with an endpoint in \(\set{a_i, a_{i+1}}\), contradiction.
\end{cproof}

\begin{claim}\label{cl:5}
$|I_1| \le m^2$.
\end{claim}
\begin{cproof}
Without loss of generality, assume that for every $i \in I_1$, the path $Q_1^i$ has an endpoint in $\{a_i, a_{i+1}\}$.  For every $i \in I_1$, fix $R_i$ to be a path from $\{a_i, a_{i+1}\} \setminus V(Q_1^i)$ to $V(H^i \cup Q_1^i \cup Q_2^i \cup Q_3^i)$ in $W_i$.
Since the linkage \(Q^i_1\), \(Q^i_2\), \(Q^i_3\) maximizes the number of endpoints in \(\set{a_i, a_{i+1}}\), the path \(R_i\) must have an endpoint in \(Q^i_1\), as otherwise we would reroute one of the paths \(Q^i_2\) or \(Q^i_3\) to  $\{a_i, a_{i+1}\} \setminus V(Q_1^i)$.
Thus, \(Q^i_1 \cup R_i\) is a tree, which contains \(a_{i}\), \(a_{i+1}\) and a vertex of \(H^i\).
Let \(P\) be an \(a_1\)--\(a_\lambda\) path in \(G - Z\) such that for every \(i \in I_1\), we have \(P[W_i \cap V(P)] \subseteq Q_1^i \cup R_i\).

Let \(F\) be the union of \(P\) and all paths \(Q^i_j\) with \(i \in I_1\) and \(j \in \set{1, 2, 3}\).
Every component of \(F\) is either a union of paths having a common endpoint in a vertex from \(Z\) but otherwise disjoint, or the union of \(P\) and paths \(Q_1^i\) with \(i \in I_1\).
Thus \(F\) has at most \(|Z|+1\) components, and each \(H^i\) with \(i \in I_1\) intersects three components of \(F\), each in one vertex.
The graph \(G\) does not have an \(L_{k+1}\) minor,
so by Lemma~\ref{lem:new-building-ladder} applied to the graphs \(H_i\) with \(i \in I_1\), we have \(|I_1| \le {(|Z| + 1)}^2 \le {m}^2\).
\end{cproof}

By Claims \ref{cl:I0bded} and \ref{cl:5} and the bound on $|I|$, we see that $|I_2| > 2k^3(m-1)$.
Without loss of generality, for all $i \in I_2$, assume that $Q_1^i$ has $a_i$ as an endpoint and $Q_2^i$ has $a_{i+1}$ as an endpoint.  It follows that $Q_3^i$ has an endpoint in $Z$.
For every $i \in I_2$, fix an \(L_k\)-model \(\phi^i\) of \(L_k\) in \(H^i\), and for each \(j \in \set{1, 2, 3}\), let $t(i, j)$ be the index such that $Q_j^i$ has an endpoint in $\phi^i((1, {t(i,j)})) \cup \phi^i((2, {t(i,j)}))$.
Thus, $t(i,j)\in \{1, \dots, k\}$ for all $i$ and $j$.

As \(|I_2| > 2k^3(m-1)\), there exist indices \(i\), \(i'\) and \(i''\) in \(I_2\) with \(i < i' < i''\) and a vertex \(z \in Z\) such that
$t(i, j) = t(i', j) = t(i'', j)$ for all $j \in \set{1, 2, 3}$,
and the paths $Q_3^i$, \(Q_3^{i'}\) and $Q_3^{i''}$ all have $z$ as an endpoint.
Using the symmetries of a ladder, we may assume that $t(i,1) \le t(i,2)$, and thus $t(i', 1) \le t(i', 2)$ and \(t(i'', 1) \le t(i'', 2)\).

There are now three cases to consider: $t(i,3) \le t(i,1)$, $t(i,1) < t(i,3) < t(i,2)$, and $t(i,2) \le t(i,3)$.  In each case, we find two paths $R_1$ and $R_2$ linking $H^i$ and $H^{i''}$ such that by joining a rooted ladder minor in $H^i$ to a rooted ladder minor in $H^{i''}$, we show that $G$ has an $L_{k+1}$ minor, yielding a contradiction.
The rooted ladder minors will be obtained by applying Lemma~\ref{lem:rooted-ladder-from-ladder}. \\

\begin{figure}[!h]
  \centering
  \includegraphics[scale=0.8]{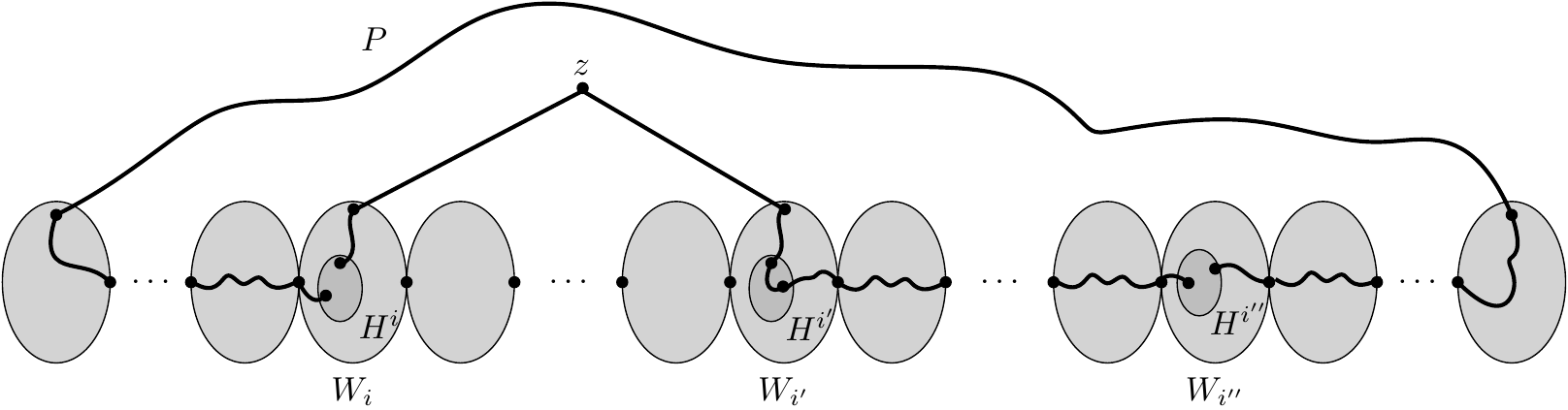}
  \caption{\label{fig:end_of_proof-case_1}
  Illustration of Case~1.
  The paths $R_1$ and $R_2$ are depicted in bold.}
\end{figure}

\textbf{Case 1: $t(i,3) \le t(i,1)$.}
See Figure~\ref{fig:end_of_proof-case_1} for an illustration of this case.
Let $R_1$ be a \(V(H^{i})\)--\(V(H^{i''})\) path contained in
\(V(Q_3^i \cup Q_3^{i'} \cup H^{i'} \cup Q_2^{i'}) \cup W_{i'+1} \cup \dots \cup W_{i''-1} \cup V(Q_1^{i''})\).
Thus, $R_1$ links a vertex of $\phi^i((1, {t(i, 3)})) \cup \phi^i((2, t(i, 3)))$ to a vertex of $\phi^{i''}((1, t(i'', 1))) \cup \phi^{i''}((2, t(i'', 1)))$ and is internally disjoint from $H^i \cup H^{i''}$.
The path \(R_1\) has only one vertex not contained in \(C\), namely \(z\). Hence by our choice of \(i_0\), there exists a path \(P\) between \(W_0 \cup \dots \cup W_{i_0-1}\) and \(W_{i_0+\lambda_0 + 1} \cup \dots \cup W_\lambda\) that is internally disjoint from \(C\) and does not contain the vertex \(z\).
Let $R_2$ be a \(V(H^{i})\)--\(V(H^{i''})\) path contained in \(V(Q_1^i) \cup W_0 \cup \dots \cup W_{i-1} \cup V(P) \cup W_{i'' + 1} \cup \dots \cup W_{\lambda} \cup V(Q^{i''}_2)\).
Thus, the path $R_2$ links a vertex of $\phi^i((1, t(i, 1))) \cup \phi^i((2, t(i, 1)))$ to a vertex of $\phi^{i''}((1, t(i'', 2))) \cup \phi^{i''}((2, t(i'', 2)))$ and is internally disjoint from $H^i \cup H^{i''}$ and completely disjoint from $R_1$. By Lemma~\ref{lem:rooted-ladder-from-ladder}, $H^i$ contains an $L_{k-t(i,1)+1}$-model rooted on the endpoints of $R_1$ and $R_2$, and $H^{i''}$ has an $L_{t(i'',1)}$-model rooted on the endpoints of $R_1$ and $R_2$.  Together, we see that $G$ contains an $L_{k+1}$ minor.\\

\textbf{Case 2: $t(i,1) < t(i,3) < t(i,2)$.}  Let $R_1$ be the path formed by the union of $Q_3^i$ and $Q_3^{i''}$.
Thus, $R_1$ links a vertex of $\phi^i((1, t(i, 3))) \cup \phi^i((2, t(i, 3)))$ to a vertex of $\phi^{i''}((1, t(i'', 3))) \cup \phi^{i''}((2, t(i'', 3)))$ and is internally disjoint from $H^i \cup H^{i''}$.
Let $R_2$ be a \(V(H^i)\)--\(V(H^{i''})\) path contained in \(V(Q_2^{i}) \cup W_{i+1} \cup \dots \cup W_{i''-1} \cup V(Q_1^{i''})\). The path $R_2$ links a vertex of $\phi^{i}((1, t(i, 2))) \cup \phi^i((2, t(i, 2)))$ to a vertex of $\phi^{i''}((1, t(i'', 1))) \cup \phi^{i''}((2, t(i'', 1)))$.  By Lemma \ref{lem:rooted-ladder-from-ladder}, $H_i$ contains an $L_{t(i, 3)}$-model rooted on the endpoints of $R_1$ and $R_2$, and $H_{i''}$ contains an $L_{k-t(i'',3)+1}$-model rooted on the endpoints of $R_1$ and $R_2$, implying that $G$ contains an $L_{k+1}$ minor.\\

\textbf{Case 3: $t(i,2) \le t(i,3)$.}  Let $R_1$ be a \(V(H^i)\)--\(V(H^{i''})\) path contained in \(V(Q_2^i) \cup W_{i+1} \cup \dots \cup W_{i'-1} \cup V(Q_1^{i'} \cup H^{i'} \cup Q_3^{i'} \cup Q_3^{i''})\).
Thus, $R_1$ links a vertex of $\phi^i((1, {t(i, 2)})) \cup \phi^i((2, t(i, 2)))$ to a vertex of $\phi^{i''}((1, t(i'', 3))) \cup \phi^{i''}((2, t(i'', 3)))$ and is internally disjoint from $H^i \cup H^{i''}$.
The path \(R_1\) has only one vertex not contained in \(C\), namely \(z\). Hence by our choice of \(i_0\), there exists a path \(P\) between \(W_0 \cup \dots \cup W_{i_0-1}\) and \(W_{i_0+\lambda_0 + 1} \cup \dots \cup W_\lambda\) which is internally disjoint from \(C\) and does not contain the vertex \(z\).
Let $R_2$ be a \(V(H^{i})\)--\(V(H^{i''})\) path contained in \(V(Q_1^i) \cup W_0 \cup \dots \cup W_{i-1} \cup V(P) \cup W_{i'' + 1} \cup \dots \cup W_{\lambda} \cup V(Q^{i''}_2)\).
Thus, the path $R_2$ links a vertex of $\phi^i((1, t(i, 1))) \cup \phi^i((2, t(i, 1)))$ to a vertex of $\phi^{i''}((1, t(i'', 2))) \cup \phi^{i''}((2, t(i'', 2)))$ and is internally disjoint from $H^i \cup H^{i''}$ and completely disjoint from $R_1$. By Lemma~\ref{lem:rooted-ladder-from-ladder}, $H^i$ contains an $L_{k-t(i,2)+1}$-model of rooted on the endpoints of $R_1$ and $R_2$, and $H^{i''}$ has an $L_{t(i'',2)+1}$-model rooted on the endpoints of $R_1$ and $R_2$.  Together, we see that $G$ contains an $L_{k+1}$ minor.\\

In each case, we showed that $G$ contains an $L_{k+1}$ minor, and thus, in each case we arrive at a contradiction to our assumptions on the graph $G$, completing the proof of the theorem.
\end{proof}

\section*{Acknowledgements}

We are much grateful to the three anonymous referees for their careful reading and very helpful comments.

\bibliographystyle{plain}
\bibliography{ladders-dimension}

\begin{thebibliography}{10}

\bibitem{AES87}
Kiyoshi Ando, Hikoe Enomoto, and Akira Saito.
\newblock Contractible edges in 3-connected graphs.
\newblock {\em Journal of Combinatorial Theory, Series B}, 42(1):87--93, 1987.

\bibitem{Diestel5thEdition}
Reinhard Diestel.
\newblock {\em Graph Theory}.
\newblock Springer Publishing Company, Incorporated, 5th edition, 2017.

\bibitem{ES35}
P.~Erd\H{o}s and G.~Szekeres.
\newblock A combinatorial problem in geometry.
\newblock {\em Compositio Math.}, 2:463--470, 1935.

\bibitem{GS21+}
Maximilian Gorsky and Michał~T. Seweryn.
\newblock Posets with $k$-outerplanar cover graphs have bounded dimension.
\newblock \href{http://arxiv.org/abs/2103.15920}{arXiv:2103.15920}.

\bibitem{H69a}
R.~Halin.
\newblock Untersuchengen \"uber minimale $n$-fach zusammenh\"angende graphen.
\newblock {\em Mathematische Annalen}, 182:175--188, 1969.

\bibitem{H69b}
R.~Halin.
\newblock Zur theorie der $n$-fach zusammenh\"agenden graphen.
\newblock {\em Abh. Math. Sem. Univ. Hamburg}, 33:133--164, 1969.

\bibitem{HSTWW19}
David~M. Howard, Noah Streib, William~T. Trotter, Bartosz Walczak, and Ruidong
  Wang.
\newblock Dimension of posets with planar cover graphs excluding two long
  incomparable chains.
\newblock {\em Journal of Combinatorial Theory, Series A}, 164:1--23, 2019.

\bibitem{JMOdMW19}
Gwena\"el Joret, Piotr Micek, Patrice~Ossona de~Mendez, and Veit Wiechert.
\newblock Nowhere dense graph classes and dimension.
\newblock {\em Combinatorica}, 39(5):1055--1079, 2019.
\newblock \href{http://arxiv.org/abs/1708.05424}{arXiv:1708.05424}.

\bibitem{JMMTWW}
Gwena\"{e}l Joret, Piotr Micek, Kevin~G. Milans, William~T. Trotter, Bartosz
  Walczak, and Ruidong Wang.
\newblock Tree-width and dimension.
\newblock {\em Combinatorica}, 36(4):431--450, 2016.
\newblock \href{http://arxiv.org/abs/1301.5271}{arXiv:1301.5271}.

\bibitem{JMTWW17}
Gwena\"el Joret, Piotr Micek, William~T. Trotter, Ruidong Wang, and Veit
  Wiechert.
\newblock On the dimension of posets with cover graphs of treewidth $2$.
\newblock {\em Order}, 34(2):185--234, 2017.
\newblock \href{http://arxiv.org/abs/1406.3397}{arXiv:1406.3397}.

\bibitem{JMW_PlanarPosets}
Gwena\"el Joret, Piotr Micek, and Veit Wiechert.
\newblock Planar posets have dimension at most linear in their height.
\newblock {\em SIAM J. Discrete Math.}, 31(4):2754--2790, 2018.
\newblock \href{http://arxiv.org/abs/1612.07540}{arXiv:1612.07540}.

\bibitem{JMW18}
Gwena\"el Joret, Piotr Micek, and Veit Wiechert.
\newblock Sparsity and dimension.
\newblock {\em Combinatorica}, 38(5):1129--1148, 2018.
\newblock \href{http://arxiv.org/abs/1507.01120}{arXiv:1507.01120}.

\bibitem{Kel81}
David Kelly.
\newblock On the dimension of partially ordered sets.
\newblock {\em Discrete Mathematics}, 35:135--156, 1981.

\bibitem{KMT21+}
Jakub Kozik, Piotr Micek, and William~T. Trotter.
\newblock Dimension is polynomial in height for posets with planar cover
  graphs.
\newblock \href{http://arxiv.org/abs/1907.00380}{arXiv:1907.00380}.

\bibitem{MW15}
Piotr Micek and Veit Wiechert.
\newblock Topological minors of cover graphs and dimension.
\newblock {\em Journal of Graph Theory}, 86(3):295--314, 2017.
\newblock \href{http://arxiv.org/abs/1504.07388}{arXiv:1504.07388}.

\bibitem{NOdM-book}
Jaroslav Ne{\v{s}}et{\v{r}}il and Patrice Ossona~de Mendez.
\newblock {\em Sparsity}, volume~28 of {\em Algorithms and Combinatorics}.
\newblock Springer, Heidelberg, 2012.
\newblock Graphs, structures, and algorithms.

\bibitem{RS86}
Neil Robertson and Paul~D. Seymour.
\newblock Graph minors. {V}. {E}xcluding a planar graph.
\newblock {\em Journal of Combinatorial Theory, Series B}, 41(1):92--114, 1986.

\bibitem{S20}
Michał~T. Seweryn.
\newblock Improved bound for the dimension of posets of treewidth two.
\newblock {\em Discrete Mathematics}, 343(1):111605, 2020.
\newblock \href{http://arxiv.org/abs/1902.01189}{arXiv:1902.01189}.

\bibitem{ST14}
Noah Streib and William~T. Trotter.
\newblock Dimension and height for posets with planar cover graphs.
\newblock {\em European J. Combin.}, 35:474--489, 2014.

\bibitem{trotter-walczak-wang}
William~T. Trotter, Bartosz Walczak, and Ruidong Wang.
\newblock Dimension and cut vertices: an application of {R}amsey theory.
\newblock In {\em Connections in {D}iscrete {M}athematics}, pages 187--199.
  Cambridge Univ. Press, Cambridge, 2018.
\newblock \href{https://arxiv.org/abs/1505.08162}{arXiv:1505.08162}.

\bibitem{Walczak17}
Bartosz Walczak.
\newblock Minors and dimension.
\newblock {\em J. Combin. Theory Ser. B}, 122:668--689, 2017.
\newblock \href{http://arxiv.org/abs/1407.4066}{arXiv:1407.4066}.

\end{thebibliography}

\end{document}